\documentclass[a4paper,12pt]{article}

\usepackage{lineno,hyperref}
\usepackage{mathrsfs}
\usepackage{bm}
\usepackage{float}
\usepackage{amssymb,amsmath,amsthm,latexsym}
\usepackage{amsfonts}
\usepackage{graphicx}
\usepackage{pdflscape}
\usepackage{setspace, subcaption}
\usepackage{mathtools}

\DeclarePairedDelimiter\floor{\lfloor}{\rfloor}
\usepackage[lined,boxruled,linesnumbered]{algorithm2e}
\usepackage{pgfplots}
\usepackage{mathrsfs}
\usepackage{rotating}

\usetikzlibrary{arrows}
\usetikzlibrary{shapes}
\newtheorem{theorem}{Theorem}[section]
\newtheorem{prop}[theorem]{Proposition}

\newtheorem{definition}[theorem]{Definition}

\newtheorem{coro}[theorem]{Corollary}
\newtheorem{conj}[theorem]{Conjecture}
\newtheorem{obs}[theorem]{Observation}
\newtheorem{construction}[theorem]{Construction}
\begin{document}
	
	\begin{center}
		{\large \bf {Efficient domination in Lattice graphs}}\\
		{\large\vspace{0.1in} 
			A. Senthil Thilak\footnote{Corresponding author}, Bharadwaj$^2$}\\
		\vspace{0.1in}
		{\small \it $^{1,2}$ Department of Mathematical and Computational Sciences \\
			National Institute of Technology Karnataka, Surathkal, \\
			Srinivasnagar - 575 025. Mangalore, India}\\
		%{\small \it $^2$Department of Mathematics \\
		% NMAM Institute of Technology, \\
		% Nitte - 574110 Udupi Dist., Karnataka, India}\\
		{\small \it E-mails: $ ^1 $asthilak23@gmail.com, 
			$ ^2 $bwajhsvj@gmail.com}
	\end{center}
	
	\begin{abstract}
		\noindent Given a graph $G$, a subset $S$ of vertices of $G$ is \textit{an efficient dominating set (\textit{EDS})} if $|N[v] \cap S|=1,$ for all $v\in V(G)$. A graph $G$ is \textit{efficiently dominatable} if it possesses an \textit{EDS}. The \textit{efficient domination number} of $G$ is denoted by $F(G)$ and is defined to be $\max \left\{\sum_{v \in S}(1 + \operatorname{deg} v):\right.$ $\left.S \subseteq V(G)\right.$ and $\left.|N[x] \cap S| \leq 1,  \forall~ x \in V(G)\right\}$. In general, not every graph is efficiently dominatable. Further, the class of efficiently dominatable graphs has not been completely characterized and the problem of determining whether or not a graph is efficiently dominatable is NP-Complete. Hence, interest is shown to study the efficient domination property for graphs under restricted conditions or special classes of graphs.  In this paper, we study the notion of efficient domination in some Lattice graphs, namely, rectangular grid graphs ($P_m \Box P_n$), triangular grid graphs, and hexagonal grid graphs. \\ 
		
		\vspace{2mm}
		
		
		\vspace{2mm}
		
		
	\end{abstract}
	
	%\thanks{This work was supported by ..}
	
	%%%%%%%%%%%%%%%%%%%%%%%%%%%%%%%%%%%%%%%%%%%%%%%%%%%%%%%%%%%%%%%%%%%
	
	%\maketitle

	\section{Introduction}
	
	\noindent Given a graph $G = (V, E)$, a set $S \subseteq V(G)$ is a \textit{dominating set} if each vertex $v \in V(G)$ is either in $S$ or has at least one neighbor in $S$. The size of the smallest dominating set of $G$ is the \textit{domination number of $G$} and is denoted by $\gamma(G)$. The \textit{open neighborhood of a vertex $v$}, denoted by $ N(v)$, is the set of all vertices adjacent to $v$ and the closed neighborhood of $v$, denoted by $ N[v]$ is defined as $ N(v) \cup \{v\}$. A set $S\subseteq V(G)$ is an \textit{efficient dominating set} (\textit{EDS}) of $G$ if $|N[v]\cap S| =1$, for all $v\in V(G)$. That is,  $S$ is an \textit{EDS}, if each vertex $v \in V(G)$ is dominated by exactly one vertex (including itself) in $S$. Not every graph possesses an \textit{EDS}. If a graph $G$ has an \textit{EDS}, then it is said to be \textit{efficiently dominatable.} \\
	
	The \textit{distance} between a pair of vertices $ u $ and $ v $ is the length of the shortest path between $ u $ and $ v $ and is denoted by $ d(u,v)$.  A set $S\subseteq V(G)$ is a \textit{2-packing} if for each pair $u,v \in S$, $N[u]\cap N[v] = \emptyset$. If $S$ is a \textit{2-packing}, then $d(u,v) \geq 3,$ for all $u,v \in S$.  Thus, a dominating set is an \textit{EDS} if and only if it is a $2$-packing. The \textit{influence} of a set $S\subseteq  V(G)$ is denoted by $ I(S) $ and is the number of vertices dominated by $ S $ (inclusive of vertices in $S$). If $S$ is a $2$-packing, then $ I(S)= \sum_{v\in S}  [1+ \deg(v)].$ The maximum influence of a $2$-packing  of $G$ is called the \textit{efficient domination number} of $G$ and is denoted by $F(G)$. That is, $F(G) = \max\{I(S):~$S is a 2-packing$\} = \max \left\{\sum_{v \in S}(1 + \operatorname{deg} v): S \subseteq V(G)\right.$ and $\left.|N[x] \cap S| \leq 1,  \forall~ x \in V(G)\right\}$. Clearly, $0 \le F(G) \le |V(G)|$ and $G$ is efficiently dominatable if and only if $F(G) = |V(G)|$.  A $2$-packing with influence $F(G)$ is called an \textit{$F(G)$-set}. \\
	
	The concept of efficient domination is found in the literature in different names like \textit{perfect codes}  or \textit{perfect $1$-codes} \cite{biggs1973}, \textit{independent perfect domination} \cite{yen1996},  \textit{perfect $ 1 $-domination} \cite{livingston1990} and \textit{efficient domination} \cite{bange1988}. In this paper, we use the terminology ``{efficient domination}'' introduced by Bange et.~al.~\cite{bange1988}. The problem of determining whether $ F(G) = |V(G)| $ is $\mathcal{NP} $-complete on arbitrary graphs \cite{bange1988} as well as on some special/restricted classes of graphs like bipartite graphs, chordal graphs, planar graphs of degree at most three, etc. [\cite{haynes1998}, whereas it is polynomial in the case of trees \cite{bange1988}. In \cite{goddard2000}, Goddard et al. have obtained bounds on the efficient domination number of arbitrary graphs and trees. Efficient Domination has also been studied on different special classes of graphs like chordal bipartite graphs \cite{tang2002}, strong product of arbitrary graphs \cite{abay2009}, cartesian product of cycles \cite{chelvam2011}, etc. Hereditary efficiently dominatable graphs were defined and studied in \cite{milanic2013} and \cite{barbosa2016}. \\
	
	The perfect codes or efficient domination finds wide applications in coding theory, resource allocation in computer networks, etc. (refer to \cite{va2001}, \cite{livingston1988}), while lattice graph structures play a significant role in source and channel coding. Motivated by the applications and the graph theoretical significance of efficient domination and lattice structures, this papers focuses on the study of efficient domination in some special lattice structures, namely, finite rectangular grids ($P_n \Box P_m$, where $ 1 \le n, m < \infty $) (in \ref{rg}), infinite rectangular grids (in \ref{irg}), infinite triangular (in \ref{itg})  and hexagonal grid graphs (in \ref{hex1}). The study on finite cases of triangular and hexagonal grid structures is in progress.  
	
	\section{Main Results} \label{sec1}
	\subsection{Notations and Terminologies}
	\begin{definition} \cite{imrich2000}
		The \textit{cartesian product} of two graphs $G = (V_1, \, E_1)$ and $H= (V_2, \, E_2)$, denoted by $G\Box H$, is the graph with vertex set  $V_{1}\times V_{2}$ in which two vertices $(u_1,v_1)$ and $ (u_2, \, v_2)$ are adjacent if and only if either (i) $u_{1}=u_{2}$ and $v_{1}v_{2}\in E_{2}$ or (ii) $u_{1}u_{2}\in E_{1}$ and $v_{1}=v_{2}$. \\
		The graphs $ G $ and $ H $ are called the \textit{factors of $ G \Box H $.} For $ v \in V(H) $, the subgraph of $ G \Box H $ induced by $ \{(u, v) \in V(G \Box H): u \in V(G)\} $ is called the \textit{G}-\textit{layer of $ G \Box H$} with respect to $ v $ and is denoted by $ G^{(v)} $.  Analogously, the $ H $-\textit{layer}, namely, $ H^{(u)} $ is defined for each $ u \in V(G). $  
	\end{definition}
	
	In literature, the cartesian product of two paths is referred by different terminologies like \textit{grid} graphs, \textit{rectangular grid graphs}, \textit{two-dimensional lattice} graphs, etc.~\cite{Acharya1981}. In this paper, we use \textit{rectangular grid graph} to refer to the carterisan product of two paths $ P_n $ and $ P_m $.  \\
	In the discussions to follow, we consider three categories of rectangular grids: 
	\begin{itemize}
		\item[(i)] Those grids bounded on all four sides, referred to as \textbf{finite rectangular grids}, denoted by $ P_n \Box P_m $, where $ 1 \le n, m < \infty $.
		\item[(ii)] Those grids bounded on three sides (top, left and right) or two sides (top and left) and unbounded on the other sides, denoted respectively as \textbf{$ \bf P_n \Box P_\infty $, where $ \bf n \ge 1 $} and \textbf{$ \bf P_\infty \Box P_\infty $}.
		\item[(iii)] Those grids unbounded on all four sides, referred to as \textbf{infinite rectangular grids.}
	\end{itemize}
	
	Throughout this paper, we refer to vertices that are not dominated by a set as \textit{voids}. In the figures given throughout this paper, \textit{shaded circular dots of larger size} represent \textit{dominating vertices}, \textit{shaded circular dots of smaller size} denote \textit{vertices dominated by a set (excluding self-dominating vertices)} and \textit{non-shaded circular dots} correspond to \textit{voids}. \\
	\subsection{Finite lattice graphs}
	
	\subsubsection{Finite rectangular grid graphs} \label{rg}
	
	It is known that if a graph is efficiently dominatable, then all its EDSs are of same size and is equal to $ \gamma(G) $ \cite{bange1978, bange1988}. The domination number of rectangular grid graphs has been studied in \cite{chang1993}, \cite{gravier1997}.  In the discussions to follow, we give constructive characterizations for efficiently dominatable lattice graphs.  In the case of graphs which are not efficiently dominatable, we obtain either the exact value or bounds of efficient domination number.
	
	%For the sake of convenience, we use $v_{i,j}$ to represent the vertex $(u_i,v_j)$. 
	In what follows, assume that $V(P_n\Box P_m) = \{v_{i,j}|1 \leq i\leq m, 1 \leq j \leq n\}$, unless mentioned otherwise. For convenience, we label the $ m $ $ P_n $-layers of $ P_n \Box P_m $ as $R_1, R_2, \dots R_m$ and its $n$ $ P_m $-layers as $C_1,C_2, \dots C_n$, respectively.  That is, for each $ i $ $(1 \le i \le m),$ $ R_i \cong P_n $ and for each $ j $ $ (1 \le j \le n),$ $ C_j \cong P_m $.  
	%The vertices of $R_i$ are  $\{v_{i,j}|1 \leq j \leq n\}$ and the vertices of  $C_i$ are  $\{v_{i,j}|1 \leq i \leq m \}$. 
	Further, it is noted that the distance between any two vertices $v_{i,j}$ and $v_{p,q}$ is $|i-p|+|j-q|$.  With these conventions, the results discussed below lead to characterizations of efficiently dominatinable rectangular grids.
	
	\begin{theorem}\label{th1}
		For $ n \ge 1 $, $P_n \Box P_2$ is efficiently dominatable if and only if $ n $ is odd.
	\end{theorem}
	
	\begin{proof}
		%    The vertices of $P_2 \Box P_n$ are $\{v_{i,j}| 1 \leq i \leq 2, 1 \leq j \leq n\}$.\\
		%    \textbf{Case(i):} 
		The result is trivially true if $ n =1 $. So, assume that $ n > 1 $ and $ n $ is odd. That is, $n=2k+1$ for some natural number $k$. Then, $\{v_{1,j}| j=1,5,\dots, 4\floor
		{\frac{k}{2}}+1 \} \cup \{(v_{2,j})| j=3,7,\dots, 4\floor
		{\frac{k-1}{2}}+3 \}$ is an \textit{EDS} of $ P_n \Box P_2 $ and hence, it is efficiently dominatable.
		\begin{figure}[!h]
			\centering
			\includegraphics[height=2cm]{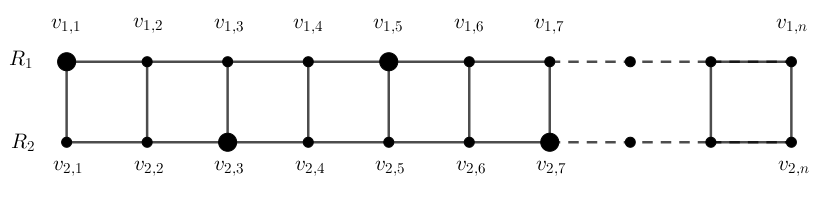}
			\caption{$P_n \Box P_2$}
			\label{Fig_PnP2}
		\end{figure}
		Suppose n is even, then we show that $ F(P_n \Box P_2) < 2n. $ Let $n=2k$, for some natural number $k$ and $S$ be an $F(P_n \Box P_2)$-set.  If possible, assume that $I(S)=2n$. Then, by definition, $ |N[v_{i, j}] \cap S| = 1 $, for each $ v_{i, j} \in V(P_n \Box P_2). $
		
		Consider an arbitrary vertex, say $ v_{1, 1} $. Since $ I(S) = 2n $, to dominate $ v_{1, 1} $, either $ v_{1,1} \in S $ or $ v_{1, 2} \in S $ or $ v_{2, 1} \in S $.  Suppose $ v_{1, 2} \in S $, then $ N[v_{2,1}] \cap S = \emptyset $, which is a contradiction.
		Hence, either $ v_{1,1} \in S $ or $ v_{2, 1} \in S $.  Without loss of generality, let $ v_{1, 1} \in S $.  Then, progressively including vertices in $ S $ (refer to figure \ref{Fig_PnP2}), it can be observed that to dominate $v_{2,2}$, $v_{2,3}$ must be in $ S $. Next, to dominate $v_{1,4}$,  $v_{1,5} \in S.$ Continuing this pattern, we get $S = \{v_{1,j}| j= 1,5,9, \dots ,4\floor{\frac{k-1}{2}}+1\} \cup \{v_{2,j}| j= 3,7,11, \dots,4\floor{\frac{k}{2}}-1\}$. If $k$ is odd, then $ S $ dominates all vertices except $v_{2,n}$. If $k$ is even, then $ S $ dominates all vertices but $v_{1,n}$. In either case, $I(S) = 2n-1$, which is a contradiction. Therefore, $ P_n \Box P_2 $ is not efficiently dominatable when $ n $ is even and in particular, $F(P_n \Box P_2) = 2n-1$.  
	\end{proof}
	
	%The blocks of $3 \times 3 $, $3 \times 6$, $3 \times 7$, and $3 \times 8$ grid along with their $F(G)-set$ is shown in figure \ref{fig18}. We call these blocks A, B, C, and D respectively.\\
	%We concatenate these blocks from left to right by drawing an edge between the end vertices as shown in figure \ref{fig20}.   $A^n-X$ represents concatenating $A$ with itself $n$ times and then concatenating $X$ at the end.  We join the corresponding end vertices by an edge.
	
	In the next theorem, to study efficient domination in $P_n \Box P_3$, the grid is partitioned columnwise into $k$ blocks, say, $B_1, B_2, \dots B_k$, such that each $ B_i $ $(1 \le i \le k-1)$ is of size $3 \times 3 $ and $B_k$ is of size either $3 \times 3$ or $3 \times 4$ or $3 \times 5$, depending on whether $n \equiv 0\,\,or\,\,1\,\,or\,\,2\,\,(mod\,\,3)$, respectively.  Here, we refer a block $B_i$ to be \textit{internal} if it is adjacent to two other blocks (namely, $B_{i-1}$ and $B_{i+1}$) and a \textit{terminal block} if it is adjacent to only one block (to the left or right).
	\begin{obs}\label{obs_P3P3}
		Clearly, as observed in figure \ref{fig_P3P3}, $ P_3 \Box P_3 $ (or any $ 3 \times 3 $ block) is not efficiently dominatable and $ F(P_3 \Box P_3) = 7 $, resulting in two voids.  And, if a $ 3 \times 3 $ block occurs as an internal block or a terminal block, then out of its two voids, atmost one can be dominated by an adjacent block in $ P_n \Box P_3 $. 
		Therefore, any $3 \times 3$ block in $P_n \Box P_3$ contains at least one void and this leads to a total of at least $\floor{\frac{n}{3}}$ voids in $ P_n \Box P_3$. 
	\end{obs}
	\begin{figure}[!h]
		\centering
		\includegraphics{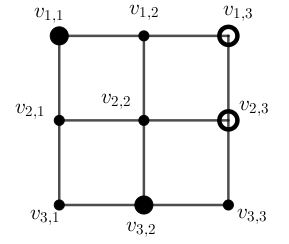}
		\caption{Efficient domination in $P_3 \Box P_3$}
		\label{fig_P3P3}
	\end{figure}
	In Theorem \ref{thm_PnP3} we show the existence of a $ 2 $-packing that results in exactly $\floor{\frac{n}{3}}$ voids in $ P_n \Box P_3$.  Thereby, it is proved that $F(P_n \Box P_3) = 3n - \floor{\frac{n}{3}}$ and hence, it is not efficiently dominatable.  Further, in Theorem \ref{thm_PnP3}, for ease of reference we follow a different labeling for the vertices of $ P_n \Box P_3 $, than the one mentioned earlier in this section.  Upon partitioning the vertices into blocks, label the vertices of each block in $ P_n \Box P_3 $ as shown in figure \ref{fig_BL}. Note that vertices at same positions in different blocks receive similar labels, but are distinguished in terms of the block they belong to. Similar pattern is extended for blocks of larger size. 
	\begin{figure}[!h]
		\centering
		\includegraphics{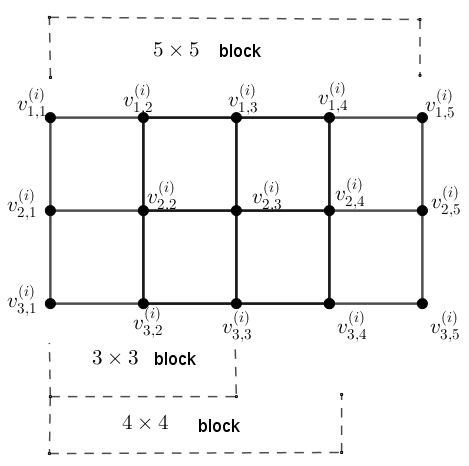}
		\caption{Labeling of vertices of block $B_i$ in $ P_n \Box P_3 $}
		\label{fig_BL}
	\end{figure}
	%We have given $P_9 \Box P_3$ along with its $F(G)-set$ in figure \ref{fig20}
	
	\begin{theorem} \label{thm_PnP3}
		$P_n \Box P_3$ is not efficiently dominatable and $F(P_n \Box P_3)=3n-\floor{\frac{n}{3}}$.
	\end{theorem}
	
	\begin{proof}
		An $ F(P_n \Box P_3)$-set is obtained by choosing vertices blockwise as follows: As explained earlier, at most two vertices can be chosen from any $ 3 \times 3 $ block.  Let $ S = \{v_{1,1}^{(i)}, v_{3,2}^{(i)}: 1 \le i \le k-2\} $.  It can be observed from figure \ref{fig_P3P12} that by the above choice of two vertices from each $ B_i\,\,(1 \le i \le k-3) $, the vertex $v_{1,3}^{(i)}$ that appeared as void in $B_i$ is later dominated by $v_{1,1}^{(i+1)}$.  This results in exactly one void, namely $v_{2,3}^{(i)}$ in each $ B_i\,\,(1 \le i \le k-3) $.
		Next, based on the value of $n$, suitable vertices are chosen from $B_{k-1}$ and $B_k$ as explained below: \\
		\textbf{Case (i):} $n \equiv 0\,\,(mod \,\,3)$\\
		Let $ S' = \{v_{2,1}^{(k-1)}, v_{1,3}^{(k-1)}\} \cup \{v_{3,1}^{(k)}, v_{2,3}^{(k)}\}$.  By this choice, the voids $v_{2,3}^{(k-2)}$, $ v_{3,3}^{(k-1)} $ and $ v_{1,1}^{(k)} $ that were in $ B_{k-2} $, $ B_{k-1} $ and $ B_{k} $ are later dominated by $v_{2,1}^{(k-1)}$, $ v_{3,1}^{(k)} $ and $ v_{1,3}^{(k-1)} $ respectively (refer to figure \ref{fig_P3P12}). So, $ S' $ results in exactly one void each in $ B_{k-2} $, $ B_{k-1} $ and $ B_{k} $. Totally there are $\floor{\frac{n}{3}}$ blocks and $ S \cup S' $ results in one void in each block.  Hence, $ S \cup S' $ is a $ 2 $-packing with influence $ 3n-\floor{\frac{n}{3}} $ in $ P_n \Box P_3 $. \\
		\begin{figure}[!h]
			\centering
			\includegraphics[width=12cm]{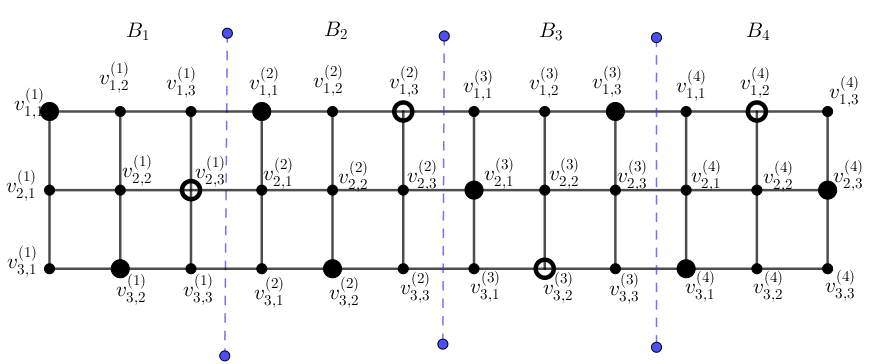}
			\caption{$P_{12} \Box P_{3}$}
			\label{fig_P3P12}
		\end{figure}\\
		\noindent \textbf{Case (ii):} $n \equiv 1 \,\,(mod\,\,3)$\\ 
		In this case, the last block $ B_k $ is of size $ 4 \times 4 $ (refer to figure \ref{fig_P3P13}).  Let $ S'' = \{v_{1,1}^{(k-1)}, v_{3,2}^{(k-1)}\} \cup \{v_{1,1}^{(k)}, v_{3,2}^{(k)}, v_{2,4}^{(k)}\}$.
		Then, by a similar argument as in case (i), $ S \cup S'' $ is a $ 2 $-packing with influence $ 3n-\floor{\frac{n}{3}} $ in $ P_n \Box P_3 $, when $n \equiv 1 \,\,(mod\,\,3)$ (refer to figure \ref{fig_P3P13}).
		\begin{figure}[!h]
			\centering
			\includegraphics[width=12cm]{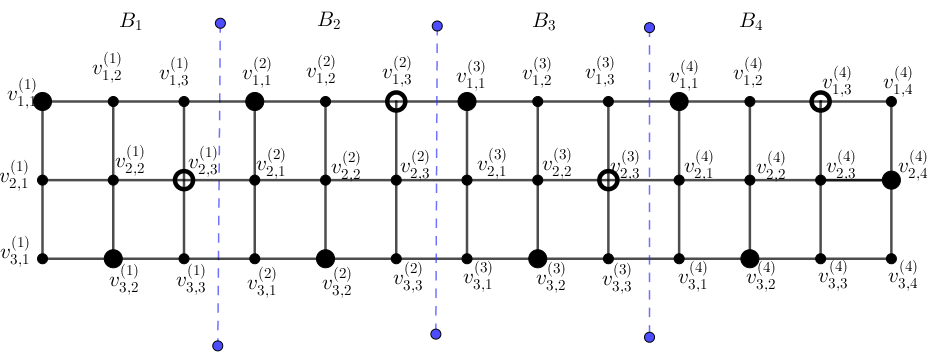}
			\caption{$P_{13} \Box P_{3}$}
			\label{fig_P3P13}
		\end{figure}\\
		
		\noindent \textbf{Case (iii):} $n \equiv 2 \,\,(mod\,\,3)$\\ 
		In this case, $ B_k $ is of size $ 3 \times 5 $.  With $ S''' = \{v_{1,1}^{(k-1)}, v_{3,2}^{(k-1)}\} \cup \{v_{1,1}^{(k)}, v_{3,2}^{(k)}, v_{1,4}^{(k)}, v_{3,5}^{(k)}\} $, it can be shown by a similar argument as in case (i) that $ S \cup S''' $ is a $ 2 $-packing with influence $ 3n-\floor{\frac{n}{3}} $ in $ P_n \Box P_3 $, when $n \equiv 2 \,\,(mod\,\,3)$ (refer to figure \ref{fig_P3P14}).
		\begin{figure}[!h]
			\centering
			\includegraphics[width=12cm]{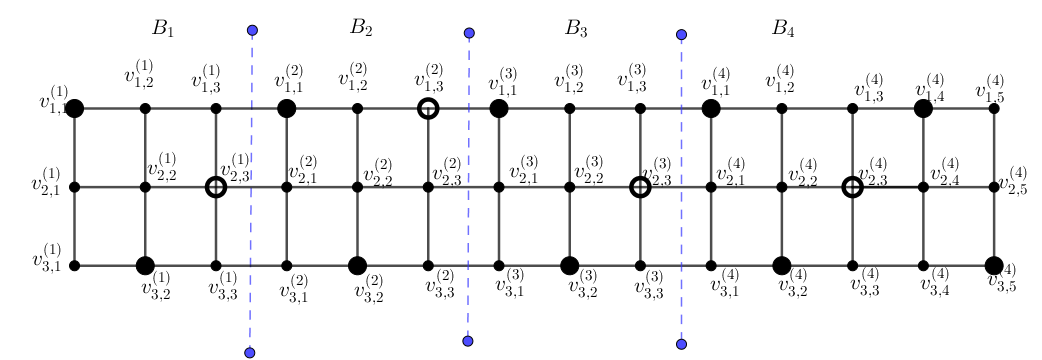}
			\caption{$P_{14} \Box P_{3}$}
			\label{fig_P3P14}
		\end{figure}\\
		Thus, in each case, $ P_n \Box P_3 $ has a $ 2 $-packing with influence $3n - \floor{\frac{n}{3}}$ and it follows from Observation \ref{obs_P3P3} that it is the maximum influence of $ P_n \Box P_3$.  Hence, the result follows.
	\end{proof}
	The next proposition deals with efficient domination in square grids of sizes $ 4 $, $ 5 $, and $ 6 $ and the results follow trivially (refer to figures \ref{fig_P4P4}, \ref{fig_P5P5} and \ref{fig_P6P6}).
%	\newpage
	\begin{prop}\label{prop_P4To6}
		\hfill
		\begin{itemize}
			\item[(i)] $ P_4 \Box P_4$ is efficiently dominatable.
			\item[(ii)] $ P_5 \Box P_5 $ is not efficiently dominatable and $ F(P_5 \Box P_5) = 23. $
			\item[(iii)] $ P_6 \Box P_6 $ is not efficiently dominatable and $ F(P_6 \Box P_6) = 33. $ 
		\end{itemize}
		%It can be observed from the figure \ref{fig12} that $P_4 \Box P_4$ is efficiently dominatable.
		%$F(P_5 \Box P_5) =23$ and $F(P_6 \Box P_6) = 33$ (refer to figure \ref{fig14} and figure \ref{fig13})\\
		%In the case of $P_m \Box P_n$,  the vertices of degree  $4$ are called internal vertices.
	\end{prop}
	%\vspace{-0.5cm}
	%\begin{center}
	\begin{figure}[!h] \centering
		\begin{subfigure}{0.45\textwidth}
			\centering
			\includegraphics[scale=0.9]{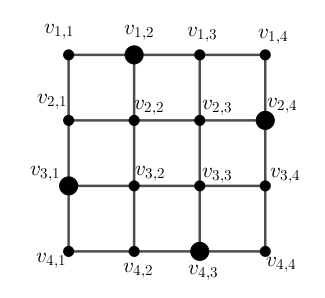}
			\caption{\centering $P_4 \Box P_4$}
			\label{fig_P4P4}
		\end{subfigure}
		%  \begin{figure}[H]
		\begin{subfigure}{0.45\textwidth}
			\includegraphics[scale=0.9]{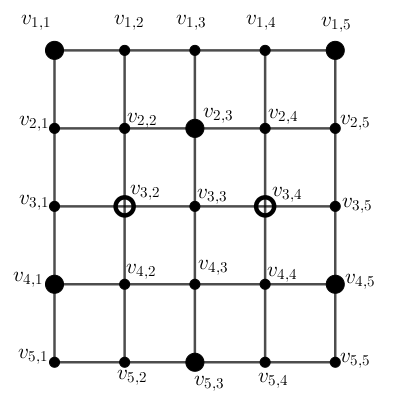}
			\caption{\centering $P_5 \Box P_5$}
			\label{fig_P5P5}
		\end{subfigure}
		\\
		\vspace{0.7cm}
		\begin{subfigure}{0.95\textwidth}
			\centering 
			\includegraphics[scale=0.9]{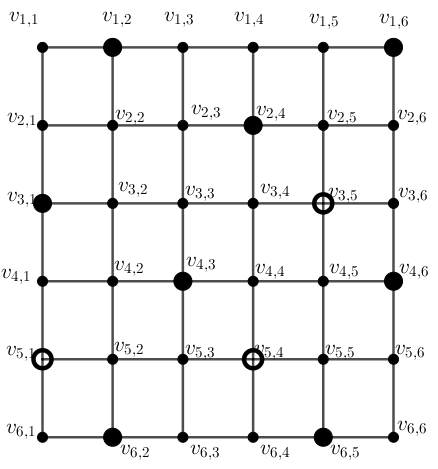}
			\caption{\centering $P_6 \Box P_6$ }
			\label{fig_P6P6}
		\end{subfigure}
		\caption{Efficient domination in square grids of sizes 4, 5, and 6}  
	\end{figure}
	%\end{center}
	%\vspace{-0.5cm}
	Next in Theorem \ref{th3}, we discuss the notion of efficient domination in $ P_n \Box P_m $, for $ n, m \ge 7 $.  Later using the proof technique adopted in this theorem and Proposition \ref{prop_P4To6}, we characterize the efficiently dominatable rectangular grids $ P_n \Box P_m $, for $ n, m \ge 3 $.  The following observation supports the discussions in Theorem \ref{th3}.
	
	\begin{obs}{\label{obs_PnPm}}	
%		Clearly, $ P_3 \Box P_3 $ (or a $ 3 \times 3 $ block) occurs as an induced subgraph of $ P_n \Box P_m $, for $ n, m \ge 7 $.  
Suppose that $ P_n \Box P_m $ ($ n, m \ge 7 $)  is partitioned into blocks of suitable sizes and we identify vertices from each block to include in an $ F(P_n \Box P_m) $-set.  Then, choosing vertices from a $ 3 \times 3 $ block as in either figure \ref{fig_Void_P3P3}(a) (choosing $ v_{i, j} $ and  $v_{i+2, j+2}$) or figure \ref{fig_Void_P3P3}(b) (choosing $ v_{i, j}$ and $ v_{i+2, j+2} $) will create a void in that block, at a vertex of degree four. In such cases, these voids cannot be further dominated by any other adjacent block in $ P_n \Box P_m $, unlike the cases discussed earlier in Theorem \ref{thm_PnP3}.  Hence, if such voids occur independently in any (sub)block of size $ 3 \times 3 $, they will continue to be voids in $ P_n \Box P_m$. 
%		In $P_5 \times P_5$ (refer to figure \ref{fig14}), we can find block A formed between $v_{2,3}$ and $v_{4,1}$ Similarly, in $P_6 \times P_6$ (refer to figure \ref{fig13}), we can find block A and block B. 
	\end{obs}
	\begin{figure}[!h]
		\centering
		\begin{subfigure}{0.45\linewidth}
			\centering
			\includegraphics[scale=0.5]{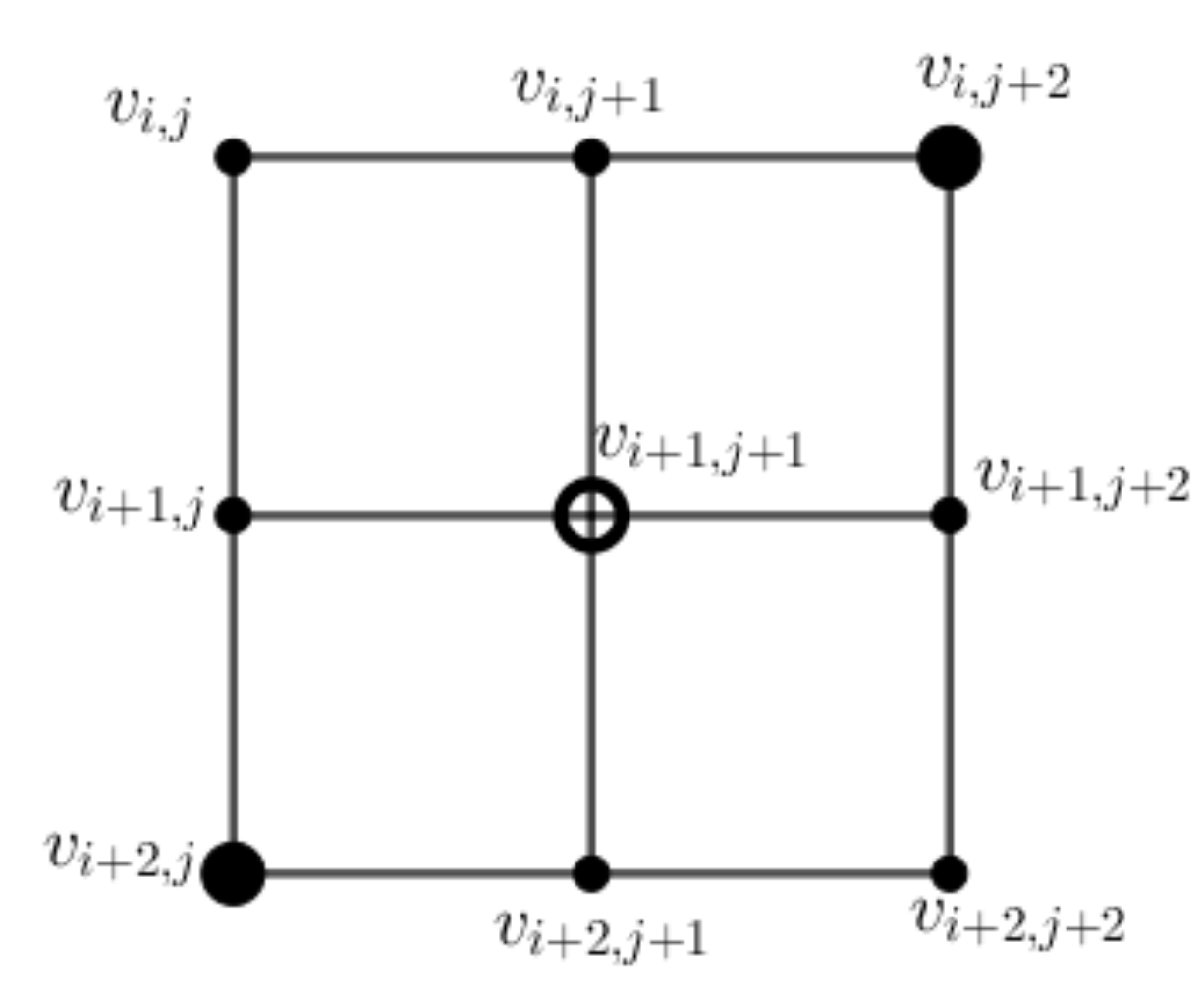}
			\caption{}
		\end{subfigure}
		\begin{subfigure}{0.45\linewidth}
			\centering
			\includegraphics[scale=0.5]{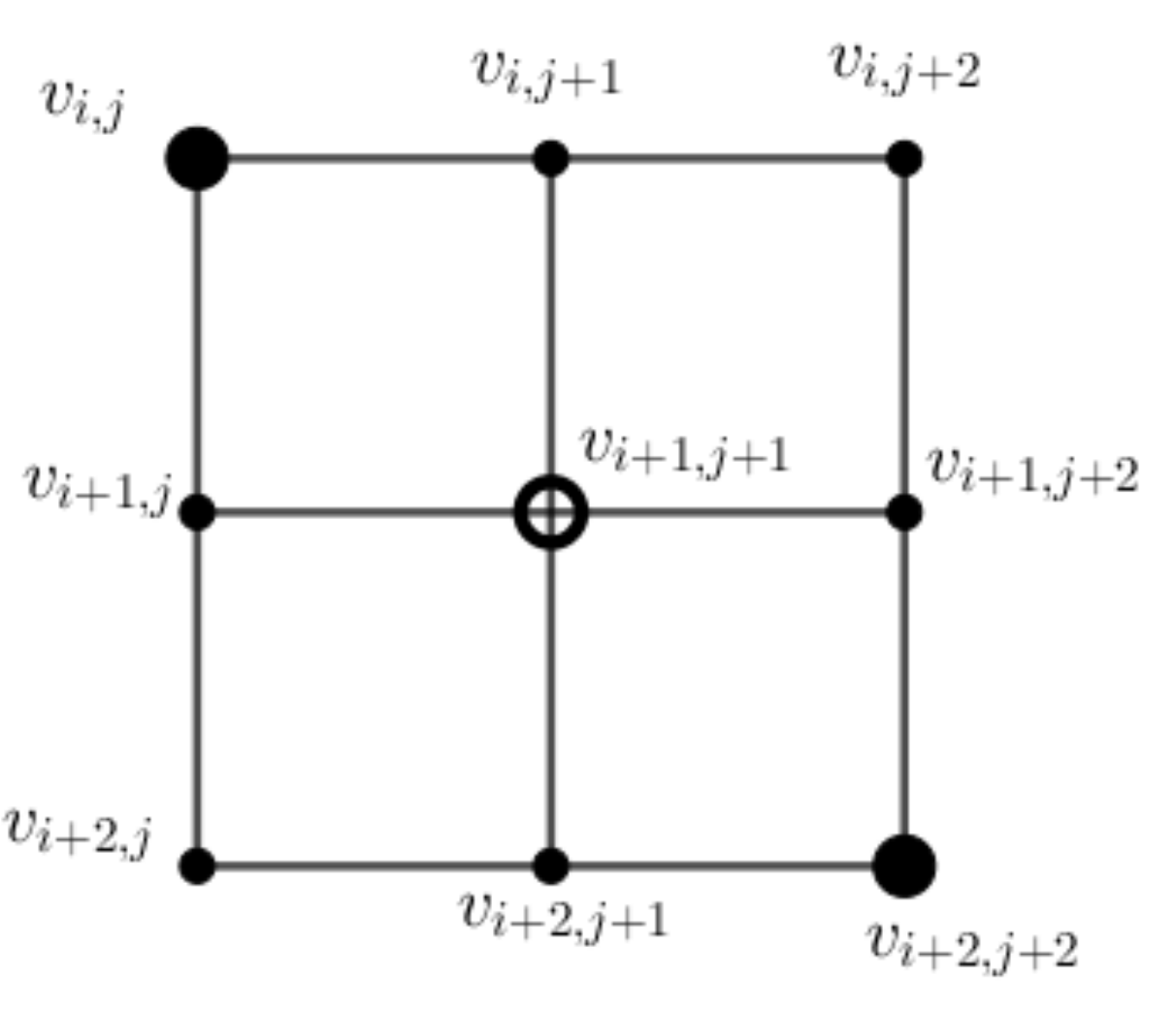}
			\caption{}
		\end{subfigure}
		\caption{$ 2 $-packings of $ P_3 \Box P_3 $ creating a void at a vertex of degree four}
		\label{fig_Void_P3P3}
	\end{figure}
	
	\begin{theorem}\label{th3}
		If $n \geq 7 $ and $ m \geq 7$, then $P_n \Box P_m$ is not efficiently dominatable.     
	\end{theorem}
	\begin{proof}
	Let $ S $ be an $F(P_n \Box P_m)$-set. If possible, assume that $I(S)=nm$. Then, $ |N[v_{i, j}] \cap S| = 1 $, for each $ v_{i, j} \in V(P_n \Box P_m)$.  Choose an arbitrary vertex, say $ v_{1, 1} $. Then, either $ v_{1,1} \in S $ or exactly one of its neighbors, namely, $ v_{1, 2} $ or $ v_{2, 1} $ must be in $ S $. The above three cases are discussed in detail as below: \\
	\noindent \textbf{Case (i):} $v_{1,1} \in S$\\
	If $ v_{1,1} \in S $, then to dominate $v_{2,2}$, either $v_{2,3}$ or $v_{3,2} $ must be in $ S $. Suppose $v_{2,3}\in S$, then to dominate $v_{3,1}$, $v_{4,1} \in S$.  But, choosing $ v_{2,3} $ and $ v_{4,1} $ to include in $ S $ results in a $ 3 \times 3 $ block in which the vertices are dominated as in figure \ref{fig_Void_P3P3} (a) (with $ i = 2,\,\,j=1 $).
	Then, it follows from Observation \ref{obs_PnPm} that $ N[v_{3,2}] \cap S = \emptyset $, which is a contradiction. \\
	On the other hand, if $v_{3,2} \in S$, then by a similar argument as above, to dominate $v_{1,3}$, $v_{1,4} \in S$. This choice of vertices results in a $ 3 \times 3 $ block with domination as in figure \ref{fig_Void_P3P3} (a) (with $ i = 1,\,\,j = 2 $).  Hence, $N[v_{2,3}] \cap S = \emptyset$, which is a contradiction.
	\begin{figure}[!h]
		\centering
		\includegraphics[]{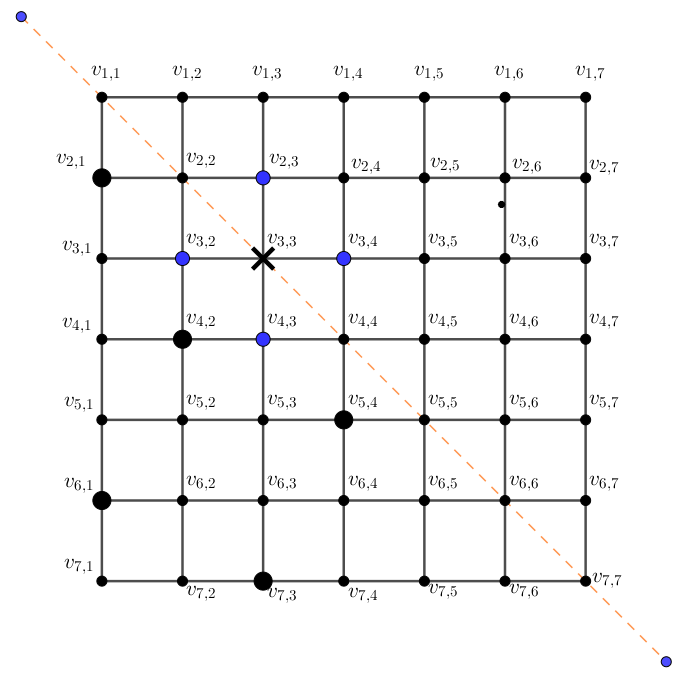}
		\caption{Efficient domination in $P_7 \Box P_7$}
		\label{fig17}
	\end{figure}\\
	\textbf{Case (ii):} $v_{2,1} \in S$\\
	To dominate $v_{4,1}$, either $v_{5,1}$ or $v_{4,2} \in S$. If $v_{5,1}\in S$, then to dominate $v_{3,2}$, $v_{3,3}$ must be in $S$. This results in a $ 3 \times 3 $ block with domination as in figure \ref{fig_Void_P3P3} (a) (with $ i = 3,\,\,j = 1 $).  Consequently, $N[v_{4,2}] \cap S = \emptyset$, leading to a contradiction. So, let $v_{4,2} \in S$. Then, to dominate $v_{5,1}$, $v_{6,1} \in S$. Next, to dominate $v_{6,3}$, either $v_{6,4}$ or $v_{7,3} \in S$. \\
	If  $v_{6,4} \in S$, then it results in a $ 3 \times 3 $ block with domination as in figure \ref{fig_Void_P3P3} (b) (with $ i = 4,\,\,j = 2 $).  Hence, $N[v_{5,3}] \cap S = \emptyset$, leading to a contradiction. \\ So, let $v_{7,3}\in S$. Then, to dominate $v_{5,3}$ we are left with only one choice, namely, $v_{5,4}$ to include in $ S $. At this stage, $S = \{v_{2,1},v_{4,2},v_{6,1},v_{7,3},v_{5,4}\}$. But, to dominate $v_{3,3}$, we are left with no choice as all its neighbors are at distance $1$ or $2$ from $S$ (refer to figure \ref{fig17}). Hence, $N[v_{3,3}] \cap S = \emptyset$, which is again a contradiction. \\
	\textbf{Case (iii):} $v_{1,2} \in S$\\
	Since there is an automorphism $f(v_{i,j})= v_{j,i}$ of $P_n \Box P_m$ that maps $v_{2,1}$ to $v_{1,2}$, the argument made in case(ii) can be modified accordingly. This results in $S = \{v_{1,2},v_{2,4},v_{1,6},v_{3,7},v_{4,5}\}$ in $ S $. But, as discussed in case (ii), to dominate $v_{3,3}$, we are left with no choice as all its neighbors are at distance $1$ or $2$ from $S$. Hence, $N[v_{3,3}] \cap S = \emptyset$, which is a contradiction. \\
	Summarizing the above arguments, it can be observed that each of these cases resulted in a void in $P_n \Box P_m$, whcih cannot be further dominated by any adjacent block.  
	In particular, such a void is created in a $ 7 \times 7 $ block of $ P_n \Box P_m$ (i.e., considering the first seven rows and the first seven columns).  As $ P_7 \Box P_7 $ is an induced subgraph of $ P_n \Box P_m $, for all $ n \ge 7 $ and $ m \ge 7 $, $ P_n \Box P_m $ is not efficiently dominatable when $ n, m \ge 7 $. 
\end{proof}	

The efficient domination in the grids $ P_n \Box P_4 $ for $ n > 4 $, $ P_n \Box P_5 $ for $ n > 5 $ and $ P_n \Box P_6 $ for $ n > 6 $ can be studied by similar arguments as in Theorem \ref{th3} and we arrive at the following result.
\begin{theorem}\label{thm_PnP4ToPnP6}
%	\begin{itemize}
%		\item[(i)] 
		For $ 4 \le m \le 6 $ and $ n > m $, $P_n \Box P_m$ is not efficiently dominatable.  
%		\item[(ii)] For $n > 5  $, $P_n \Box P_5$ is not efficiently dominatable.  
%		\item[(iii)] For $n > 6  $, $P_n \Box P_6$ is not efficiently dominatable.  
%	\end{itemize}
	   
\end{theorem}

The results discussed above in Theorem \ref{th1} to Theorem \ref{thm_PnP4ToPnP6} lead to the following characterization for efficiently dominatable grid graphs.
\begin{coro} \label{cor1}
	If $n \geq 3$ and $ m \geq 3$, then $P_n \Box P_m$ is efficiently dominatable if and only if $n=m=4$.
\end{coro}

Next, using Construction \ref{Const_PnPn} discussed below, we derive a lower bound on $F(P_n \Box P_n) $ for $n\geq 7$. The bound is obtained by constructing a $2$-packing which dominates all vertices of $ P_n \Box P_n $, except a few vertices at the boundaries. Interestingly, it is evident from Table \ref{Table_Voids_PnPn} that the $ 2 $-packing obtained in the construction is nearly optimal (that is, most likely an $ F(P_n \Box P_n)$-set), as equality in the derived lower bound is attained for most values of $ n $. In addition, the construction helps in generalizing the efficient domination property in the infinite cases disucssed in Section $ \ref{InfLat} $. An illustration of the construction is shown in figure \ref{fig15} for $ P_9 \Box P_9 $ and it is easy to extend the construction for any $ n \ge 7 $.
\begin{figure}[!h]
	\definecolor{qqqqcc}{rgb}{0.,0.,0.8}
	\begin{tikzpicture}[scale=1.25,line cap=round,line join=round,>=triangle 45,x=1.0cm,y=1.0cm]
	\clip(3.71796,-7.65852806198719) rectangle (23.114644886003077,-1.6281968275611618);
	\draw [line width=1.6pt] (7.,-3.)-- (7.,-7.);
	\draw [line width=1.6pt] (7.,-3.)-- (11.,-3.);
	\draw [line width=1.6pt] (11.,-3.)-- (11.,-7.);
	\draw [line width=1.6pt] (7.,-7.)-- (11.,-7.);
	\draw [line width=1.6pt] (7.,-3.5)-- (11.,-3.5);
	\draw [line width=1.6pt] (7.,-4.)-- (11.,-4.);
	\draw [line width=1.6pt] (7.,-4.5)-- (11.,-4.5);
	\draw [line width=1.6pt] (7.,-5.)-- (11.,-5.);
	\draw [line width=1.6pt] (7.,-5.5)-- (11.,-5.5);
	\draw [line width=1.6pt] (7.,-6.)-- (11.,-6.);
	\draw [line width=1.6pt] (7.,-6.5)-- (11.,-6.5);
	\draw [line width=1.6pt] (7.5,-3.)-- (7.5,-7.);
	\draw [line width=1.6pt] (8.,-3.)-- (8.,-7.);
	\draw [line width=1.6pt] (8.5,-3.)-- (8.5,-7.);
	\draw [line width=1.6pt] (9.,-3.)-- (9.,-7.);
	\draw [line width=1.6pt] (9.5,-3.)-- (9.5,-7.);
	\draw [line width=1.6pt] (10.,-3.)-- (10.,-7.);
	\draw [line width=1.6pt] (10.5,-3.)-- (10.5,-7.);
	\draw (6.527876398564419,-2.408787956059439) node[anchor=north west] {$v_{1,1}$};
	\draw (10.713796841984866,-2.4575749015905815) node[anchor=north west,xshift=10pt,yshift=-10pt] {$v_{1,9}$};
	\draw (6.420545105143382,-6.926459112243219) node[anchor=north west,yshift=-2pt] {$v_{9,1}$};
	\draw (10.947974209448947,-6.994760835986818) node[anchor=north west] {$v_{9,9}$};
	\draw [line width=1.6pt,dash pattern=on 2pt off 2pt,color=qqqqcc] (7.,-3.5)-- (9.,-2.5);
	\draw [line width=1.6pt,dash pattern=on 2pt off 2pt,color=qqqqcc] (6.5,-5.)-- (11.5,-2.5);
	\draw [line width=1.6pt,dash pattern=on 2pt off 2pt,color=qqqqcc] (6.,-6.5)-- (12.,-3.5);
	\draw [line width=1.6pt,dash pattern=on 2pt off 2pt,color=qqqqcc] (6.5,-7.5)-- (11.5,-5.);
	\draw [line width=1.6pt,dash pattern=on 2pt off 2pt,color=qqqqcc] (9.,-7.5)-- (12.,-6.);
	\begin{scriptsize}
	\draw [fill=black] (7.,-3.) circle (2.5pt);
	\draw [fill=black] (7.5,-3.) circle (2.5pt);
	\draw [fill=black] (8.,-3.) circle (4.0pt);
	\draw [fill=black] (8.5,-3.) circle (2.5pt);
	\draw [color=black] (9.,-3.) circle (4.0pt);
	\draw [fill=black] (9.5,-3.) circle (2.5pt);
	\draw [fill=black] (10.,-3.) circle (2.5pt);
	\draw [fill=black] (10.5,-3.) circle (4.0pt);
	\draw [fill=black] (11.,-3.) circle (2.5pt);
	\draw [fill=black] (7.,-3.5) circle (4.0pt);
	\draw [fill=black] (7.5,-3.5) circle (2.5pt);
	\draw [fill=black] (8.,-3.5) circle (2.5pt);
	\draw [fill=black] (8.5,-3.5) circle (2.5pt);
	\draw [fill=black] (9.,-3.5) circle (2.5pt);
	\draw [fill=black] (9.5,-3.5) circle (4.0pt);
	\draw [fill=black] (10.,-3.5) circle (2.5pt);
	\draw [fill=black] (10.5,-3.5) circle (2.5pt);
	\draw [fill=black] (11.,-3.5) circle (2.5pt);
	\draw [fill=black] (7.,-4.) circle (2.5pt);
	\draw [fill=black] (7.5,-4.) circle (2.5pt);
	\draw [fill=black] (8.,-4.) circle (2.5pt);
	\draw [fill=black] (8.5,-4.) circle (4.0pt);
	\draw [fill=black] (9.,-4.) circle (2.5pt);
	\draw [fill=black] (9.5,-4.) circle (2.5pt);
	\draw [fill=black] (10.,-4.) circle (2.5pt);
	\draw [fill=black] (10.5,-4.) circle (2.5pt);
	\draw [fill=black] (11.,-4.) circle (4.0pt);
	\draw [fill=black] (7.,-4.5) circle (2.5pt);
	\draw [fill=black] (7.5,-4.5) circle (4.0pt);
	\draw [fill=black] (8.,-4.5) circle (2.5pt);
	\draw [fill=black] (8.5,-4.5) circle (2.5pt);
	\draw [fill=black] (9.,-4.5) circle (2.5pt);
	\draw [fill=black] (9.5,-4.5) circle (2.5pt);
	\draw [fill=black] (10.,-4.5) circle (4.0pt);
	\draw [fill=black] (10.5,-4.5) circle (2.5pt);
	\draw [fill=black] (11.,-4.5) circle (2.5pt);
	\draw [color=black] (7.,-5.) circle (4.0pt);
	\draw [fill=black] (7.5,-5.) circle (2.5pt);
	\draw [fill=black] (7.,-5.5) circle (2.5pt);
	\draw [fill=black] (7.,-6.) circle (4.0pt);
	\draw [fill=black] (7.,-6.5) circle (2.5pt);
	\draw [fill=black] (7.,-7.) circle (2.5pt);
	\draw [fill=black] (7.5,-5.5) circle (2.5pt);
	\draw [fill=black] (7.5,-6.) circle (2.5pt);
	\draw [fill=black] (7.5,-6.5) circle (2.5pt);
	\draw [fill=black] (7.5,-7.) circle (4.0pt);
	\draw [fill=black] (8.,-5.) circle (2.5pt);
	\draw [fill=black] (8.,-5.5) circle (4.0pt);
	\draw [fill=black] (8.,-6.) circle (2.5pt);
	\draw [fill=black] (8.,-6.5) circle (2.5pt);
	\draw [fill=black] (8.,-7.) circle (2.5pt);
	\draw [fill=black] (8.5,-5.) circle (2.5pt);
	\draw [fill=black] (9.,-5.) circle (4.0pt);
	\draw [fill=black] (9.5,-5.) circle (2.5pt);
	\draw [fill=black] (10.,-5.) circle (2.5pt);
	\draw [fill=black] (10.5,-5.) circle (2.5pt);
	\draw [color=black] (11.,-5.) circle (4.0pt);
	\draw [fill=black] (8.5,-5.5) circle (2.5pt);
	\draw [fill=black] (9.,-5.5) circle (2.5pt);
	\draw [fill=black] (9.5,-5.5) circle (2.5pt);
	\draw [fill=black] (10.,-5.5) circle (2.5pt);
	\draw [fill=black] (10.5,-5.5) circle (4.0pt);
	\draw [fill=black] (11.,-5.5) circle (2.5pt);
	\draw [fill=black] (8.5,-6.) circle (2.5pt);
	\draw [fill=black] (9.,-6.) circle (2.5pt);
	\draw [fill=black] (9.5,-6.) circle (4.0pt);
	\draw [fill=black] (10.,-6.) circle (2.5pt);
	\draw [fill=black] (10.5,-6.) circle (2.5pt);
	\draw [fill=black] (11.,-6.) circle (2.5pt);
	\draw [fill=black] (8.5,-6.5) circle (4.0pt);
	\draw [fill=black] (9.,-6.5) circle (2.5pt);
	\draw [fill=black] (9.5,-6.5) circle (2.5pt);
	\draw [fill=black] (10.,-6.5) circle (2.5pt);
	\draw [fill=black] (10.5,-6.5) circle (2.5pt);
	\draw [fill=black] (11.,-6.5) circle (4.0pt);
	\draw [fill=black] (11.,-7.) circle (2.5pt);
	\draw [fill=black] (10.5,-7.) circle (2.5pt);
	\draw [fill=black] (10.,-7.) circle (4.0pt);
	\draw [fill=black] (9.5,-7.) circle (2.5pt);
	\draw [color=black] (9.,-7.) circle (4.0pt);
	\draw [fill=black] (8.5,-7.) circle (2.5pt);
	\draw [fill=black] (9.,-2.5) circle (2.5pt);
	\draw [fill=black] (6.5,-5.) circle (2.5pt);
	\draw [fill=black] (11.5,-2.5) circle (2.5pt);
	\draw [fill=black] (6.,-6.5) circle (2.5pt);
	\draw [fill=black] (12.,-3.5) circle (2.5pt);
	\draw [fill=black] (6.5,-7.5) circle (2.5pt);
	\draw [fill=black] (11.5,-5.) circle (2.5pt);
	\draw [fill=black] (9.,-7.5) circle (2.5pt);
	\draw [fill=black] (12.,-6.) circle (2.5pt);
	\end{scriptsize}
	\end{tikzpicture}	
	\caption{Efficient domination in $P_9 \Box P_9$}
	\label{fig15}
\end{figure}
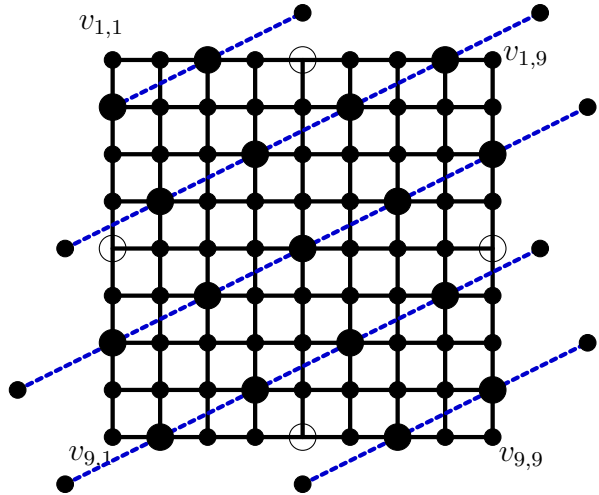
\begin{construction}{\label{Const_PnPn}}
	For $ n \ge 7$, we obtain a nearly optimal $ 2 $-packing for $P_n \Box P_n$ as follows:
	\begin{itemize}
		\item[(i)] Initially, select vertices from the first and second columns alternatively and at each selection, pick up a pair of vertices, namely, $ (v_{i, 1}, v_{i+2, 2}) $. Depending on the value of $n$, we start from either first or second row. Precisely, if $n=5k+4\,\,(k \in \mathbb{N})$, start with $ v_{2,1} $. Else, start with $ v_{1,1} $.
	\item[(ii)] Upon choosing each pair, skip two rows inbetween and proceed with the selection of next pair. 		
	For example, if $ n = 5k + 4, $ the first pair is with $(v_{2,1}, v_{4,2})$ and leaving two rows inbetwen, the second pair is $ (v_{7,1}, v_{9, 2}) $, third pair is $ (v_{12,1}, v_{14, 2}) $ and so on (refer to figure \ref{fig15}). Continue this selection until all rows are covered.  Note that the last choice may be either a pair of vertices or a single vertex in first column.
	\item[(iii)] Let $x$ be the last vertex (may be from first column or second column) chosen in the above process. Based on the choice of $ x $, we select a vertex $y$ from the last row as below:\\
	\begin{tabular}{ll}
		Case(i)& : \text{Suppose $x$ is $v_{n-2,2}$, then $y=v_{n,3}$.} \\ 
		Case(ii)& : \text{Suppose $x$ is $v_{n-1,2}$, then $y=v_{n,5}$.}\\
		Case(iii)& : \text{Suppose $x$ is $v_{n-1,1}$, then $y=v_{n,4}$.}\\
		Case(iv)& : \text{Suppose $x$ is $v_{n,1}$, then $y=x=v_{n,1}$.}\\
		Case(v)& : \text{Suppose $x$ is $v_{n,2}$, then $y=x=v_{n,2}$.} {\footnotesize (Occurs when $ n = 5k + 4 $)}
	\end{tabular}
	\item[(iv)] Upon fixing $ y $ as above, select the vertices on the last row which are at a distance of $\{5k|k=1,2,\dots,\floor{\frac{n}{5}}\}$ from $y$.
	\item[(v)] Next, for each $v_{i,j}$ selected in above steps, choose the vertices \linebreak $\{v_{i-k,j+2k}: k \in \mathbb{N}\}$ until the boundary is reached. The choice is analogous to a knight placed at $v_{i,j}$ moving towards east (two-step right, one step up) repeatedly.
\end{itemize}
It is evident from the above choice of vertices that the set of vertices generated at the end forms a $ 2 $-packing of $ P_n \Box P_n $, where $ n \ge 7 $ (refer to figure \ref{fig15}).
\end{construction}
\noindent \textbf{Number of voids generated by the above $ 2 $-packing:}\\

To compute the number of voids generated by the above $ 2 $-packing, the following observations are noted:
\begin{itemize}
	\item The vertices are included in $ S $ in such a way that all vertices of $ P_n \Box P_n $ are dominated exaclty once by $ S $, except a few that lie on the boundaries. Hence, voids occur only at the boundaries, that is, on the rows $ R_1 $, $ R_n $ and the columns $ C_1 $, $ C_n $.
	\item It follows from the construction that such voids occur either between a pair of vertices in $ S $ which are at distance five apart or at corners or at distance two from corner vertices.
	\item For instance, if $ v_{1,j},v_{1,q} \in S \cap R_1$, then $ d(v_{1,j}, v_{1,q}) = 5 $. Since $ v_{1,j}$ and $v_{1,q} $ dominate their respective neighbors in $ R_1 $, out of the four internal vertices lying on the path between $ v_{1,j}$ and $v_{1,q} $, possibly, there are at most two voids between $ v_{1,j} $ and $ v_{1, q} $.  But, in cases where one of their neighbors in the adjacent row, namely, $ R_2 $ is in $ S $, then the number of voids reduces to one. Similar arguments hold for $ R_n $, $ C_1 $ and $ C_n $.
	The number of such pairs of vertices at distance five on the boundaries and consequently, the number of voids depends on the value of $ n $, as discussed in detail below:
%	By the construction \ref{Const_PnPn}, it can be observed that the voids occur at the edges or boundaries, that is on rows $R_1$ and $R_n$, on columns $C_1$ and $C_n$. A void occurs between any two selected vertices on each of these boundaries, And in some cases, corner vertices and vertices which are within a distance of $2$ from these corner vertices can be voids. Depending on the value of n, the number of voids can be observed as follows\\ 
	
	\item \textbf{Case (i):} $ n \equiv 0\,\,(mod\,\,5) $ or $n=5k$\\
	In this case, $k$ vertices from each of $C_1$ and $R_n$ belong to $ S $, resulting in a total of $2k$ voids on $C_1$ and $R_n$. Similarly, $k$ vertices from each of $C_n$ and $R_1$ belong to $ S $, resulting in a total of $2k$ voids. Hence, if $ n = 5k $, then $ S $ generates $4k$ voids in $ P_n \Box P_n $, for $ n \ge 7 $.
	\item \textbf{Case (ii):} $ n \equiv 1\,\,(mod\,\,5) $ or $n=5k+1$\\
	In this case, $k+1$ vertices from each of $R_1$,$C_1$,$R_n$ and $C_n$ belong to $ S $, resulting in $ k $ voids on each. Hence, totally $4k$ voids are generated by $ S $, if $ n= 5k +1 $.
	\item \textbf{Case (iii):} $ n \equiv 2\,\,(mod\,\,5) $ or $n=5k+2$ \\
	In this case, $k+1$ vertices from each of $R_1$ and $C_1$ belong to $ S $, generating $k$ voids on each. And, $k$ vertices from $R_n$ and $C_n$ belong to $ S $, leading to a total of $2k+1$ voids. Hence, totally $4k+1=n-k-1$ voids are generated by $ S $, when $ n = 5k + 1 $.
	\item \textbf{Case (iv):} $ n \equiv 3\,\,(mod\,\,5) $ or $n=5k+3$ \\
	Here, $k+1$ vertices from each of $R_1$, $R_n$ and $C_1$ belong to $ S $, resulting in a total of $3k+1$ of voids.  Further, $k$ vertices from $C_n$ are in $ S $, leading to $k+1$ voids. Hence, a total of $4k+2 = n-k-1$ voids are generated by $ S $.
	\item \textbf{Case (v):} $ n \equiv 4\,\,(mod\,\,5) $ or $n=5k+4$\\
	In this case, $k+1$ vertices from each of $R_1$,$C_1$,$R_n$ and $C_n$ belong to $ S $. This results in $k$ voids on each and hence, a total of $4k$ voids are generated by $ S $.\\
\end{itemize}
For $ n = 9 $, the pattern of selection is shown in figure \ref{fig15} and it can be observed that the voids occur at the edges or boundaries. 
%{\color{red}{Because any vertex of degree four is a vertex of a $2 \times 3$ grid whose opposite corner vertices are in $S$.}} 
Following the above discussion, Table \ref{Table_Voids_PnPn} gives the number of voids $(n^2 - I(S))$ for $ 7 \le n \le 22 $, where $ S $ is the $ 2 $-packing obtained using Construction \ref{Const_PnPn}. In fact, for each $ 7 \le n \le 22 $, it is observed that the influence of $ S $ obtained above is maximum and is equal to $ F(P_n \Box P_n) $. \\
\vspace{-0.2cm}
	\begin{table}[H]
		\centering
		\caption{Number of voids in $ P_n \Box P_n $}
		\vspace{-0.2cm}
%		\begin{figure}[!h]
			\includegraphics{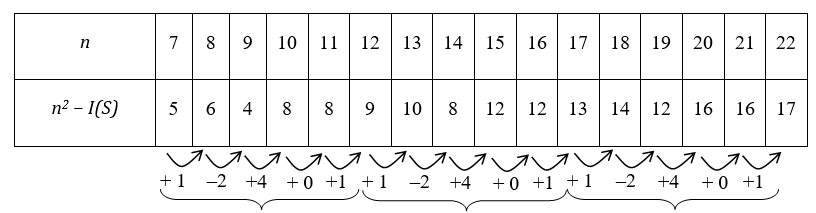}
%		\end{figure}
	\label{Table_Voids_PnPn}
	\end{table}
%		\begin{tabular}{ |c|c|c|c|c|c|c|c|c|c|c|c|c|c|c } 
%			\hline
%			$n$ & 7 & 8 & 9 & 10 & 11 & 12 & 13 & 14 & 15 & 16 & 17 \\ 
%			\hline
%			$(n^2 - I(S))$ & 5 & 6 & 4 & 8 & 8 & 9 & 10 & 8 & 12 & 12 & 13\\ 
%			\hline
%			
%			\hline
%			
%		\end{tabular}
	 Construction \ref{Const_PnPn} guarantees the existence of a $ 2 $-packing for $ P_n \Box P_n $ \linebreak $(n \ge 7) $ resulting in the number of voids as discussed above.  This leads to the following lower bound for $ P_n \Box P_n $, when $ n \ge 7 $.
	
	\begin{theorem} \label{th2}For $ n \ge 7 $ and $k=\floor{\frac{n}{5}}$, \hfill \\
		
		\qquad
		$F(P_{n}\Box P_n)\geq \begin{cases} n^2-4k; & \mbox{if $n\equiv$ 0 or 1 or 4 (mod 5)}\\
		{n^2 -n +k+1}; & \mbox{if $n\equiv$ 2 or 3 (mod 5)}
		\end{cases}$
	\end{theorem}
	As mentioned earlier, it is observed that the bound given in Thorem \ref{th2} is attained for most values of $n\geq 7$.  Based on this, we state the following conjecture.
	\begin{conj} \label{cj1} For $ n \ge 7 $ and $k=\floor{\frac{n}{5}}$, \hfill \\

		\qquad $F(P_{n}\Box P_n) = \begin{cases} n^2-4k; & \mbox{if $n\equiv$ 0 or 1 or 4 (mod 5)}\\
		{n^2 -n +k+1}; & \mbox{if $n\equiv$ 2 or 3 (mod 5)}
		\end{cases}$
	\end{conj}
	\subsection{Infinite lattice graphs} \label{InfLat}
%	The rectangular grid which has no boundary in all four directions is \textit{an infinite rectangular grid}. 
	The construction of a $2$-packing discussed for a finite rectangular grid in the previous section resulted in voids at the boundaries. The vertices included in the $ 2 $-packing lie on the diagonal lines as shown in figure \ref{fig15}. This pattern can also be extended for an infinite rectangular grid and for an infinite triangular grid. The infinite hexagonal grid has another interesting pattern which is discussed in Section \ref{hex1}.
	
	\subsubsection{Infinite Rectangular grid} \label{irg}
	As mentioned earlier, a rectangular grid that is bounded on the three sides (top, left and right) and unbounded at the bottom is referred to as $P_n \Box P_\infty $, where $ 1 \le n < \infty $. The one which is bounded on the two sides (top and left) and unbounded at right and bottom is referred to as $P_\infty \Box P_\infty $. 
	
	It is noted that Table \ref{Table_Voids_PnPn} depicts a pattern in the difference between the number of voids for consecutive values of $ n $ as follows:  $(+1,-2,+4,0,+1)$, \linebreak $(+1,-2,4,0,+1)$ $ \dots$
	Hence, as $ n $ increases or as $ n \rightarrow \infty $, the number of voids keep oscillating and does not coverge to zero.  Consequently, 
	the grids $ P_n \Box P_{\infty} $ for $ 1 \le n \le \infty $ and $ P_{\infty} \Box P_{\infty} $ are not efficiently dominatable.\\
	
	In the next result, extending construction \ref{Const_PnPn}, we prove that an infinite rectangular grid (unbounded on all four sides) is efficiently dominatable.
	\begin{theorem}
		An infinite rectangular grid is 
		efficiently dominatable.
	\end{theorem}
	
	\begin{proof}
		
		Let $G$ denote an infinite rectangular grid. Construct an EDS $S$ for $G$ as follows: Start with an arbitrary vertex, say, $v_{i, j}$ and let $S = \{v_{i, j}\}$. Next, add the four vertices $v_{i+1, j-2},v_{i+2, j+1}, v_{i-1, j+2}$, $v_{i-2, j-1}$ to $S$, if they are already not in $S$. Since $d(v_{i,j}, v_{p,q}) = |i-p|+|j-q|$, it can be observed that all these four vertices are at distance three from $v_{i, j}$ and they are also mutually at distance at least three. Hence, $ S = S \cup \{v_{i+1, j-2},v_{i+2, j+1}, v_{i-1, j+2}, v_{i-2, j-1}\}$ is a $ 2 $-packing of $ G $.  Next, for each vertex $ v_{p, q} $ in $ S $, select another set of four vertices in the same manner as above and add them to $S$. The vertices are chosen in such a way that distance between any two vertices in $ S $ is at least three and hence, the set $S$ so obtained forms a $2$-packing of $ G $.  Note that the vertices in $ S $ lie on the diagonal lines as shown in figure \ref{fig_P11P11}. Pairing the vertices of $ S $ lying on consecutive diagonal lines to form opposite corners of $2 \times 3$ grids as in figure \ref{fig2}, it can be observed these $2 \times 3$ grids are disjoint and there are no voids between the diagonal lines. 
		Since $ G $ is infinite (or unbounded on all sides), this pattern of adding vertices to $ S $ shall continue iteratively so that all vertices of $ G $ are dominated, resulting in no void.  Hence, the set $ S $ so obtained is an \textit{EDS} of $ G $ or equivalently, $ G $ is efficiently dominatable.	
	\end{proof}
%	\newpage
	\definecolor{cqcqcq}{rgb}{0.7529411764705882,0.7529411764705882,0.7529411764705882}
	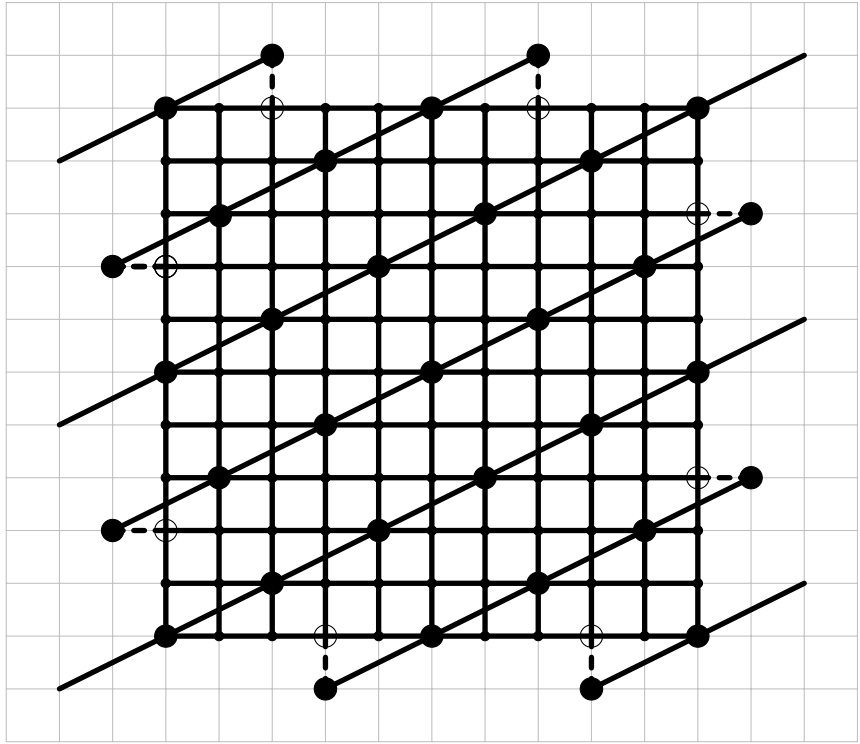
\begin{figure}[!h]		
		\begin{tikzpicture}[scale=0.7,line cap=round,line join=round,>=triangle 45,x=1.0cm,y=1.0cm]
		\draw [color=cqcqcq,, xstep=1.0cm,ystep=1.0cm] (-1.,-9.) grid (15.,5.);
		\clip(-3.,-9.) rectangle (15.,5.);
		\draw [line width=2.pt] (2.,3.)-- (12.,3.);
		\draw [line width=2.pt] (2.,2.)-- (12.,2.);
		\draw [line width=2.pt] (2.,1.)-- (12.,1.);
		\draw [line width=2.pt] (2.,0.)-- (12.,0.);
		\draw [line width=2.pt] (2.,-1.)-- (12.,-1.);
		\draw [line width=2.pt] (2.,-2.)-- (12.,-2.);
		\draw [line width=2.pt] (2.,-3.)-- (12.,-3.);
		\draw [line width=2.pt] (2.,-4.)-- (12.,-4.);
		\draw [line width=2.pt] (2.,-5.)-- (12.,-5.);
		\draw [line width=2.pt] (2.,-6.)-- (12.,-6.);
		\draw [line width=2.pt] (2.,-7.)-- (12.,-7.);
		\draw [line width=2.pt] (2.,3.)-- (2.,-7.);
		\draw [line width=2.pt] (3.,3.)-- (3.,-7.);
		\draw [line width=2.pt] (4.,3.)-- (4.,-7.);
		\draw [line width=2.pt] (5.,3.)-- (5.,-7.);
		\draw [line width=2.pt] (6.,3.)-- (6.,-7.);
		\draw [line width=2.pt] (7.,3.)-- (7.,-7.);
		\draw [line width=2.pt] (8.,3.)-- (8.,-7.);
		\draw [line width=2.pt] (9.,3.)-- (9.,-7.);
		\draw [line width=2.pt] (10.,3.)-- (10.,-7.);
		\draw [line width=2.pt] (11.,3.)-- (11.,-7.);
		\draw [line width=2.pt] (12.,3.)-- (12.,-7.);
		\draw [line width=2.pt] (1.,0.)-- (9.,4.);
		\draw [line width=2.pt] (0.,-3.)-- (14.,4.);
		\draw [line width=2.pt] (1.,-5.)-- (13.,1.);
		\draw [line width=2.pt] (0.,-8.)-- (14.,-1.);
		\draw [line width=2.pt] (5.,-8.)-- (13.,-4.);
		\draw [line width=2.pt] (0.,2.)-- (4.,4.);
		\draw [line width=2.pt] (10.,-8.)-- (14.,-6.);
		\draw [line width=2.pt,dash pattern=on 4pt off 4pt] (2.,0.)-- (1.,0.);
		\draw [line width=2.pt,dash pattern=on 4pt off 4pt] (4.,3.)-- (4.,4.);
		\draw [line width=2.pt,dash pattern=on 4pt off 4pt] (9.,3.)-- (9.,4.);
		\draw [line width=2.pt,dash pattern=on 4pt off 4pt] (12.,1.)-- (13.,1.);
		\draw [line width=2.pt,dash pattern=on 4pt off 4pt] (12.,-4.)-- (13.,-4.);
		\draw [line width=2.pt,dash pattern=on 4pt off 4pt] (10.,-7.)-- (10.,-8.);
		\draw [line width=2.pt,dash pattern=on 4pt off 4pt] (5.,-7.)-- (5.,-8.);
		\draw [line width=2.pt,dash pattern=on 4pt off 4pt] (2.,-5.)-- (1.,-5.);
		\draw [line width=2.pt,dash pattern=on 4pt off 4pt] (2.,0.)-- (1.,0.);
		\draw [line width=2.pt,dash pattern=on 4pt off 4pt] (2.,0.)-- (1.,0.);
		\begin{scriptsize}
		\draw [fill=black] (2.,3.) circle (6pt);
		\draw [fill=black] (3.,3.) circle (2.5pt);
		%	\draw[color=black] (3.1252622090220394,3.339941426571216) node {$I$};
		\draw [color=black] (4.,3.) circle (6.0pt);
		%	\draw[color=black] (4.124940846226796,3.4531125930472264) node {$J$};
		\draw [fill=black] (5.,3.) circle (2.5pt);
		%	\draw[color=black] (5.1246194834315535,3.339941426571216) node {$K$};
		\draw [fill=black] (6.,3.) circle (2.5pt);
		%	\draw[color=black] (6.124298120636309,3.339941426571216) node {$L$};
		\draw [fill=black] (7.,3.) circle (6.0pt);
		%	\draw[color=black] (7.123976757841067,3.4531125930472264) node {$M$};
		\draw [fill=black] (8.,3.) circle (2.5pt);
		%	\draw[color=black] (8.123655395045823,3.339941426571216) node {$N$};
		\draw [color=black] (9.,3.) circle (6.0pt);
		%	\draw[color=black] (9.12333403225058,3.4531125930472264) node {$O$};
		\draw [fill=black] (10.,3.) circle (2.5pt);
		%	\draw[color=black] (10.123012669455338,3.339941426571216) node {$P$};
		\draw [fill=black] (11.,3.) circle (2.5pt);
		%	\draw[color=black] (11.122691306660094,3.339941426571216) node {$Q$};
		\draw [fill=black] (12.,3.) circle (6pt);
		%	\draw[color=black] (12.141231804944187,3.4531125930472264) node {$R$};
		\draw [fill=black] (2.,2.) circle (2.5pt);
		\draw [fill=black] (2.,1.) circle (2.5pt);
		%	\draw[color=black] (2.1255835718172826,1.340584152161696) node {$T$};
		\draw [color=black] (2.,0.) circle (6.0pt);
		%	\draw[color=black] (2.1255835718172826,0.4540766814329461) node {$U$};
		\draw [fill=black] (2.,-1.) circle (2.5pt);
		%	\draw[color=black] (2.1255835718172826,-0.6587731222478247) node {$V$};
		\draw [fill=black] (2.,-2.) circle (6pt);
		%	\draw[color=black] (2.1255835718172826,-1.5452805929765743) node {$W$};
		\draw [fill=black] (2.,-3.) circle (2.5pt);
		%	\draw[color=black] (2.1255835718172826,-2.658130396657345) node {$Z$};
		\draw [fill=black] (2.,-4.) circle (2.5pt);
		%	\draw[color=black] (2.17273822451562,-3.6106543811637675) node {$A_1$};
		\draw [color=black] (2.,-5.) circle (6.0pt);
		%	\draw[color=black] (2.17273822451562,-4.497161851892517) node {$B_1$};
		\draw [fill=black] (2.,-6.) circle (2.5pt);
		%	\draw[color=black] (2.17273822451562,-5.610011655573288) node {$C_1$};
		\draw [fill=black] (3.,2.) circle (2.5pt);
		%	\draw[color=black] (3.172416861720377,2.387417442064794) node {$D_1$};
		\draw [fill=black] (3.02,0.96) circle (6pt);
		%	\draw[color=black] (3.191278722799712,1.463186249177374) node {$E_1$};
		\draw [fill=black] (3.,0.) circle (2.5pt);
		%	\draw[color=black] (3.172416861720377,0.3880601676552733) node {$F_1$};
		\draw [fill=black] (3.,-1.) circle (2.5pt);
		%	\draw[color=black] (3.172416861720377,-0.6116184695494868) node {$G_1$};
		\draw [fill=black] (3.,-2.) circle (2.5pt);
		%	\draw[color=black] (3.172416861720377,-1.611297106754247) node {$H_1$};
		\draw [fill=black] (3.,-3.) circle (2.5pt);
		\draw [fill=black] (3.,-4.) circle (6 pt);
		%	\draw[color=black] (3.172416861720377,-2.610975743959007) node {$I_1$};
		%	\draw[color=black] (3.172416861720377,-3.497483214687757) node {$J_1$};
		\draw [fill=black] (3.,-5.) circle (2.5pt);
		%	\draw[color=black] (3.172416861720377,-4.610333018368528) node {$K_1$};
		\draw [fill=black] (3.,-6.) circle (2.5pt);
		%	\draw[color=black] (3.172416861720377,-5.610011655573288) node {$L_1$};
		\draw [fill=black] (4.,2.) circle (2.5pt);
		%	\draw[color=black] (4.172095498925134,2.387417442064794) node {$M_1$};
		\draw [fill=black] (5.,2.) circle (6pt);
		%	\draw[color=black] (5.171774136129891,2.5005886085408044) node {$N_1$};
		\draw [fill=black] (6.,2.) circle (2.5pt);
		%	\draw[color=black] (6.171452773334647,2.387417442064794) node {$O_1$};
		\draw [fill=black] (7.,2.) circle (2.5pt);
		%	\draw[color=black] (7.171131410539404,2.387417442064794) node {$P_1$};
		\draw [fill=black] (8.,2.) circle (2.5pt);
		%	\draw[color=black] (8.17081004774416,2.387417442064794) node {$Q_1$};
		\draw [fill=black] (9.,2.) circle (2.5pt);
		%	\draw[color=black] (9.170488684948918,2.387417442064794) node {$R_1$};
		\draw [fill=black] (10.,2.) circle (6pt);
		%	\draw[color=black] (10.170167322153675,2.5005886085408044) node {$S_1$};
		\draw [fill=black] (11.,2.) circle (2.5pt);
		%	\draw[color=black] (11.169845959358431,2.387417442064794) node {$T_1$};
		\draw [fill=black] (12.,2.) circle (2.5pt);
		\draw [fill=black] (4.,1.) circle (2.5pt);
		%	\draw[color=black] (4.172095498925134,1.3877388048600336) node {$V_1$};
		\draw [fill=black] (5.,1.) circle (2.5pt);
		%	\draw[color=black] (5.171774136129891,1.3877388048600336) node {$W_1$};
		\draw [fill=black] (6.,1.) circle (2.5pt);
		%	\draw[color=black] (6.171452773334647,1.3877388048600336) node {$Z_1$};
		\draw [fill=black] (7.,1.) circle (2.5pt);
		%	\draw[color=black] (7.171131410539404,1.3877388048600336) node {$A_2$};
		\draw [fill=black] (8.,1.) circle (6pt);
		%	\draw[color=black] (8.17081004774416,1.500909971336044) node {$B_2$};
		\draw [fill=black] (9.,1.) circle (2.5pt);
		%	\draw[color=black] (9.170488684948918,1.3877388048600336) node {$C_2$};
		\draw [fill=black] (10.,1.) circle (2.5pt);
		%	\draw[color=black] (10.170167322153675,1.3877388048600336) node {$D_2$};
		\draw [fill=black] (11.,1.) circle (2.5pt);
		%	\draw[color=black] (11.169845959358431,1.3877388048600336) node {$E_2$};
		\draw [color=black] (12.,1.) circle (6.0pt);
		%	\draw[color=black] (12.188386457642524,1.500909971336044) node {$F_2$};
		\draw [fill=black] (4.,0.) circle (2.5pt);
		%	\draw[color=black] (4.172095498925134,0.3880601676552733) node {$G_2$};
		\draw [fill=black] (5.,0.) circle (2.5pt);
		%	\draw[color=black] (5.171774136129891,0.3880601676552733) node {$H_2$};
		\draw [fill=black] (6.,0.) circle (6pt);
		%	\draw[color=black] (6.171452773334647,0.501231334131284) node {$I_2$};
		\draw [fill=black] (7.,0.) circle (2.5pt);
		%	\draw[color=black] (7.171131410539404,0.3880601676552733) node {$J_2$};
		\draw [fill=black] (8.,0.) circle (2.5pt);
		%	\draw[color=black] (8.17081004774416,0.3880601676552733) node {$K_2$};
		\draw [fill=black] (9.,0.) circle (2.5pt);
		%	\draw[color=black] (9.170488684948918,0.3880601676552733) node {$L_2$};
		\draw [fill=black] (10.,0.) circle (2.5pt);
		%	\draw[color=black] (10.170167322153675,0.3880601676552733) node {$M_2$};
		\draw [fill=black] (11.,0.) circle (6pt);
		%	\draw[color=black] (11.169845959358431,0.501231334131284) node {$N_2$};
		\draw [fill=black] (12.,0.) circle (2.5pt);
		%	\draw[color=black] (12.188386457642524,0.3880601676552733) node {$O_2$};
		\draw [fill=black] (4.,-1.) circle (6pt);
		%	\draw[color=black] (4.172095498925134,-0.49844730307347634) node {$P_2$};
		\draw [fill=black] (5.,-1.) circle (2.5pt);
		%	\draw[color=black] (5.171774136129891,-0.6116184695494868) node {$Q_2$};
		\draw [fill=black] (6.,-1.) circle (2.5pt);
		%	\draw[color=black] (6.171452773334647,-0.6116184695494868) node {$R_2$};
		\draw [fill=black] (7.,-1.) circle (2.5pt);
		%	\draw[color=black] (7.171131410539404,-0.6116184695494868) node {$S_2$};
		\draw [fill=black] (8.,-1.) circle (2.5pt);
		%	\draw[color=black] (8.17081004774416,-0.6116184695494868) node {$T_2$};
		\draw [fill=black] (9.,-1.) circle (6pt);
		%	\draw[color=black] (9.170488684948918,-0.49844730307347634) node {$U_2$};
		\draw [fill=black] (10.,-1.) circle (2.5pt);
		%	\draw[color=black] (10.170167322153675,-0.6116184695494868) node {$V_2$};
		\draw [fill=black] (11.,-1.) circle (2.5pt);
		%	\draw[color=black] (11.169845959358431,-0.6116184695494868) node {$W_2$};
		\draw [fill=black] (12.,-1.) circle (2.5pt);
		%	\draw[color=black] (12.188386457642524,-0.6116184695494868) node {$Z_2$};
		\draw [fill=black] (4.,-2.) circle (2.5pt);
		%	\draw[color=black] (4.172095498925134,-1.611297106754247) node {$A_3$};
		\draw [fill=black] (5.,-2.) circle (2.5pt);
		%	\draw[color=black] (5.171774136129891,-1.611297106754247) node {$B_3$};
		\draw [fill=black] (6.,-2.) circle (2.5pt);
		%	\draw[color=black] (6.171452773334647,-1.611297106754247) node {$C_3$};
		\draw [fill=black] (7.,-2.) circle (6pt);
		%	\draw[color=black] (7.171131410539404,-1.4981259402782365) node {$D_3$};
		\draw [fill=black] (8.,-2.) circle (2.5pt);
		%	\draw[color=black] (8.17081004774416,-1.611297106754247) node {$E_3$};
		\draw [fill=black] (9.,-2.) circle (2.5pt);
		%	\draw[color=black] (9.170488684948918,-1.611297106754247) node {$F_3$};
		\draw [fill=black] (10.,-2.) circle (2.5pt);
		%	\draw[color=black] (10.170167322153675,-1.611297106754247) node {$G_3$};
		\draw [fill=black] (11.,-2.) circle (2.5pt);
		%	\draw[color=black] (11.169845959358431,-1.611297106754247) node {$H_3$};
		\draw [fill=black] (12.,-2.) circle (6pt);
		%	\draw[color=black] (12.188386457642524,-1.4981259402782365) node {$I_3$};
		\draw [fill=black] (4.,-3.) circle (2.5pt);
		%	\draw[color=black] (4.172095498925134,-2.610975743959007) node {$J_3$};
		\draw [fill=black] (5.,-3.) circle (6pt);
		%	\draw[color=black] (5.171774136129891,-2.4978045774829964) node {$K_3$};
		\draw [fill=black] (6.,-3.) circle (2.5pt);
		%	\draw[color=black] (6.171452773334647,-2.610975743959007) node {$L_3$};
		\draw [fill=black] (7.,-3.) circle (2.5pt);
		%	\draw[color=black] (7.171131410539404,-2.610975743959007) node {$M_3$};
		\draw [fill=black] (8.,-3.) circle (2.5pt);
		%	\draw[color=black] (8.17081004774416,-2.610975743959007) node {$N_3$};
		\draw [fill=black] (9.,-3.) circle (2.5pt);
		%	\draw[color=black] (9.170488684948918,-2.610975743959007) node {$O_3$};
		\draw [fill=black] (10.,-3.) circle (6pt);
		%	\draw[color=black] (10.170167322153675,-2.4978045774829964) node {$P_3$};
		\draw [fill=black] (11.,-3.) circle (2.5pt);
		%	\draw[color=black] (11.169845959358431,-2.610975743959007) node {$Q_3$};
		\draw [fill=black] (12.,-3.) circle (2.5pt);
		%	\draw[color=black] (12.188386457642524,-2.610975743959007) node {$R_3$};
		\draw [fill=black] (4.,-4.) circle (2.5pt);
		%	\draw[color=black] (4.172095498925134,-3.6106543811637675) node {$S_3$};
		\draw [fill=black] (5.,-4.) circle (2.5pt);
		%	\draw[color=black] (5.171774136129891,-3.6106543811637675) node {$T_3$};
		\draw [fill=black] (6.,-4.) circle (2.5pt);
		%	\draw[color=black] (6.171452773334647,-3.6106543811637675) node {$U_3$};
		\draw [fill=black] (7.,-4.) circle (2.5pt);
		%	\draw[color=black] (7.171131410539404,-3.6106543811637675) node {$V_3$};
		\draw [fill=black] (8.,-4.) circle (6pt);
		%	\draw[color=black] (8.17081004774416,-3.497483214687757) node {$W_3$};
		\draw [fill=black] (9.,-4.) circle (2.5pt);
		%	\draw[color=black] (9.170488684948918,-3.6106543811637675) node {$Z_3$};
		\draw [fill=black] (10.,-4.) circle (2.5pt);
		%	\draw[color=black] (10.170167322153675,-3.6106543811637675) node {$A_4$};
		\draw [fill=black] (11.,-4.) circle (2.5pt);
		%	\draw[color=black] (11.169845959358431,-3.6106543811637675) node {$B_4$};
		\draw [color=black] (12.,-4.) circle (6.0pt);
		%	\draw[color=black] (12.188386457642524,-3.497483214687757) node {$C_4$};
		\draw [fill=black] (4.,-5.) circle (2.5pt);
		%	\draw[color=black] (4.172095498925134,-4.610333018368528) node {$D_4$};
		\draw [fill=black] (5.,-5.) circle (2.5pt);
		%	\draw[color=black] (5.171774136129891,-4.610333018368528) node {$E_4$};
		\draw [fill=black] (6.,-5.) circle (6pt);
		%	\draw[color=black] (6.171452773334647,-4.497161851892517) node {$F_4$};
		\draw [fill=black] (7.,-5.) circle (2.5pt);
		%	\draw[color=black] (7.171131410539404,-4.610333018368528) node {$G_4$};
		\draw [fill=black] (8.,-5.) circle (2.5pt);
		%	\draw[color=black] (8.17081004774416,-4.610333018368528) node {$H_4$};
		\draw [fill=black] (9.,-5.) circle (2.5pt);
		%	\draw[color=black] (9.170488684948918,-4.610333018368528) node {$I_4$};
		\draw [fill=black] (10.,-5.) circle (2.5pt);
		%	\draw[color=black] (10.170167322153675,-4.610333018368528) node {$J_4$};
		\draw [fill=black] (11.,-5.) circle (6pt);
		%	\draw[color=black] (11.169845959358431,-4.497161851892517) node {$K_4$};
		\draw [fill=black] (12.,-5.) circle (2.5pt);
		%	\draw[color=black] (12.188386457642524,-4.610333018368528) node {$L_4$};
		\draw [fill=black] (4.,-6.) circle (6pt);
		%	\draw[color=black] (4.285266665401144,-5.496840489097277) node {$M_4$};
		\draw [fill=black] (5.,-6.) circle (2.5pt);
		%	\draw[color=black] (5.171774136129891,-5.610011655573288) node {$N_4$};
		\draw [fill=black] (6.,-6.) circle (2.5pt);
		%	\draw[color=black] (6.171452773334647,-5.610011655573288) node {$O_4$};
		\draw [fill=black] (7.,-6.) circle (2.5pt);
		%	\draw[color=black] (7.171131410539404,-5.610011655573288) node {$P_4$};
		\draw [fill=black] (8.,-6.) circle (2.5pt);
		%	\draw[color=black] (8.17081004774416,-5.610011655573288) node {$Q_4$};
		\draw [fill=black] (9.,-6.) circle (6pt);
		%	\draw[color=black] (9.170488684948918,-5.496840489097277) node {$R_4$};
		\draw [fill=black] (10.,-6.) circle (2.5pt);
		%	\draw[color=black] (10.170167322153675,-5.610011655573288) node {$S_4$};
		\draw [fill=black] (11.,-6.) circle (2.5pt);
		%	\draw[color=black] (11.169845959358431,-5.610011655573288) node {$T_4$};
		\draw [fill=black] (12.,-6.) circle (2.5pt);
		%	\draw[color=black] (12.188386457642524,-5.610011655573288) node {$U_4$};
		\draw [fill=black] (2.,-7.) circle (6pt);
		%	\draw[color=black] (2.17273822451562,-6.496519126302037) node {$V_4$};
		\draw [fill=black] (3.,-7.) circle (2.5pt);
		%	\draw[color=black] (3.172416861720377,-6.609690292778048) node {$W_4$};
		\draw [fill=black] (4.,-7.) circle (2.5pt);
		%	\draw[color=black] (4.172095498925134,-6.609690292778048) node {$Z_4$};
		\draw [color=black] (5.,-7.) circle (6.0pt);
		%	\draw[color=black] (5.171774136129891,-6.496519126302037) node {$A_5$};
		\draw [fill=black] (6.,-7.) circle (2.5pt);
		%	\draw[color=black] (6.171452773334647,-6.609690292778048) node {$B_5$};
		\draw [fill=black] (7.,-7.) circle (6pt);
		%	\draw[color=black] (7.171131410539404,-6.496519126302037) node {$C_5$};
		\draw [fill=black] (8.,-7.) circle (2.5pt);
		%	\draw[color=black] (8.17081004774416,-6.609690292778048) node {$D_5$};
		\draw [fill=black] (9.,-7.) circle (2.5pt);
		%	\draw[color=black] (9.170488684948918,-6.609690292778048) node {$E_5$};
		\draw [color=black] (10.,-7.) circle (6.0pt);
		%	\draw[color=black] (10.170167322153675,-6.496519126302037) node {$F_5$};
		\draw [fill=black] (11.,-7.) circle (2.5pt);
		%	\draw[color=black] (11.169845959358431,-6.609690292778048) node {$G_5$};
		\draw [fill=black] (12.,-7.) circle (6pt);
		%	\draw[color=black] (12.188386457642524,-6.496519126302037) node {$H_5$};
		%	\draw[color=black] (7.048529313523726,-0.13064101202644193) node {$i$};
		%	\draw[color=black] (7.048529313523726,-1.130319649231202) node {$j$};
		%	\draw[color=black] (7.048529313523726,-2.1299982864359626) node {$k$};
		%	\draw[color=black] (7.048529313523726,-3.1296769236407225) node {$l$};
		%	\draw[color=black] (7.048529313523726,-4.129355560845483) node {$m$};
		%	\draw[color=black] (7.048529313523726,-5.129034198050243) node {$n$};
		%	\draw[color=black] (7.048529313523726,-6.128712835255003) node {$p$};
		%	\draw[color=black] (7.048529313523726,-7.1283914724597635) node {$q$};
		%	\draw[color=black] (1.7483463502305818,-1.8282085091666007) node {$r$};
		%	\draw[color=black] (2.748024987435339,-1.8282085091666007) node {$s$};
		%	\draw[color=black] (3.7477036246400957,-1.8282085091666007) node {$t$};
		%	\draw[color=black] (4.747382261844853,-1.8282085091666007) node {$a$};
		%	\draw[color=black] (5.747060899049609,-1.8282085091666007) node {$b$};
		%	\draw[color=black] (6.746739536254365,-1.8282085091666007) node {$c$};
		%	\draw[color=black] (7.746418173459123,-1.8282085091666007) node {$d$};
		%	\draw[color=black] (8.74609681066388,-1.8282085091666007) node {$e$};
		%	\draw[color=black] (9.792930100566974,-1.781053856468263) node {$f_1$};
		%	\draw[color=black] (10.792608737771731,-1.781053856468263) node {$g_1$};
		%	\draw[color=black] (11.792287374976487,-1.781053856468263) node {$h_1$};
		\draw [fill=black] (1.,0.) circle (2.5pt);
		%	\draw[color=black] (1.1730595873108633,0.3880601676552733) node {$I_5$};
		\draw [fill=black] (9.,4.) circle (6 pt);
		%	\draw[color=black] (9.170488684948918,4.3867747164743145) node {$J_5$};
		%	\draw[color=black] (5.228359719367895,1.9535946372400865) node {$i_1$};
		%	\draw[color=black] (7.227716993777409,0.4446457508932786) node {$j_1$};
		\draw [fill=black] (1.,-5.) circle (6 pt);
		%	\draw[color=black] (1.1730595873108633,-4.610333018368528) node {$M_5$};
		\draw [fill=black] (13.,1.) circle (6 pt);
		%	\draw[color=black] (13.18806509484728,1.3877388048600336) node {$N_5$};
		%	\draw[color=black] (7.227716993777409,-2.0451199115789542) node {$k_1$};
		%	\draw[color=black] (7.227716993777409,-4.553747435130522) node {$l_1$};
		\draw [fill=black] (5.,-8.) circle (6 pt);
		%	\draw[color=black] (5.171774136129891,-7.609368929982808) node {$Q_5$};
		\draw [fill=black] (13.,-4.) circle (6 pt);
		%	\draw[color=black] (13.18806509484728,-3.6106543811637675) node {$R_5$};
		%	\draw[color=black] (9.227074268186923,-6.043834460397995) node {$m_1$};
		\draw [fill=black] (4.,4.) circle (6 pt);
		%	\draw[color=black] (4.172095498925134,4.3867747164743145) node {$T_5$};
		%	\draw[color=black] (2.2293238077536253,2.9532732744448467) node {$n_1$};
		\draw [fill=black] (10.,-8.) circle (6 pt);
		%	\draw[color=black] (10.170167322153675,-7.609368929982808) node {$U_5$};
		%	\draw[color=black] (12.226110179801193,-7.06237495868209) node {$p_1$};
		%	\draw[color=black] (1.588020531056234,0.5200931952106189) node {$q_1$};
		%	\draw[color=black] (4.398437831877154,3.7077477176182505) node {$r_1$};
		%	\draw[color=black] (12.603347401387895,0.916192277876656) node {$t_1$};
		%	\draw[color=black] (12.603347401387895,-4.082200908147145) node {$a_1$};
		%	\draw[color=black] (9.792930100566974,-7.288717291634112) node {$b_1$};
		%	\draw[color=black] (4.79453691454319,-7.288717291634112) node {$c_1$};
		%	\draw[color=black] (1.588020531056234,-4.4782999908131815) node {$d_1$};
		\draw [color=black] (2.,0.) circle (6.0pt);
		%	\draw[color=black] (2.17273822451562,0.501231334131284) node {$U_{6}$};
		\draw [fill=black] (1.,0.) circle (6 pt);
		%	\draw[color=black] (1.1730595873108633,0.3880601676552733) node {$I_{6}$};
		%	\draw[color=black] (1.588020531056234,0.5200931952106189) node {$q_{2}$};
		\draw [color=black] (2.,0.) circle (6.0pt);
		%	\draw[color=black] (2.17273822451562,0.501231334131284) node {$U_{7}$};
		\draw [fill=black] (1.,0.) circle (2.5pt);
		%	\draw[color=black] (1.1730595873108633,0.3880601676552733) node {$I_{7}$};
		%	\draw[color=black] (1.588020531056234,0.5200931952106189) node {$q_{3}$};
		\end{scriptsize}
		\end{tikzpicture}
		\caption{Efficient domination in an Infinite rectangular grid  }\label{fig_P11P11}\end{figure}
	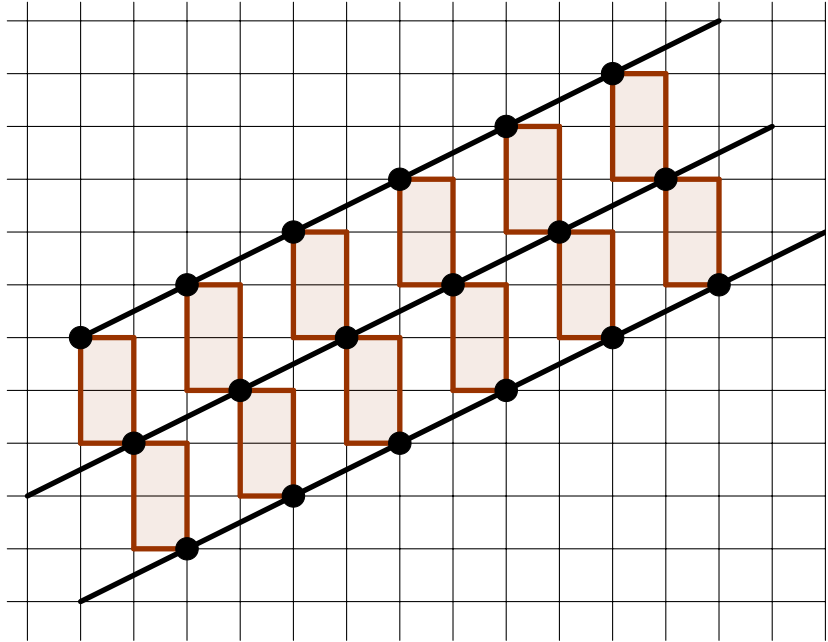
\begin{figure}[!h]
			\definecolor{zzttqq}{rgb}{0.6,0.2,0.}
			\begin{tikzpicture}[scale=0.7,line cap=round,line join=round,>=triangle 45,x=1.0cm,y=1.0cm]
			\draw [color=black,, xstep=1.0cm,ystep=1.0cm] (20.62123875396314,-11.734687997774342) grid (36.,0.3534496776321836);
			\clip(18.62123875396314,-11.734687997774342) rectangle (36.,0.3534496776321836);
			\fill[line width=2.pt,color=zzttqq,fill=zzttqq,fill opacity=0.10000000149011612] (24.,-5.) -- (24.,-7.) -- (25.,-7.) -- (25.,-5.) -- cycle;
			\fill[line width=2.pt,color=zzttqq,fill=zzttqq,fill opacity=0.10000000149011612] (26.,-4.) -- (26.,-6.) -- (27.,-6.) -- (27.,-4.) -- cycle;
			\fill[line width=2.pt,color=zzttqq,fill=zzttqq,fill opacity=0.10000000149011612] (23.,-8.) -- (23.,-10.) -- (24.,-10.) -- (24.,-8.) -- cycle;
			\fill[line width=2.pt,color=zzttqq,fill=zzttqq,fill opacity=0.10000000149011612] (22.,-6.) -- (22.,-8.) -- (23.,-8.) -- (23.,-6.) -- cycle;
			\fill[line width=2.pt,color=zzttqq,fill=zzttqq,fill opacity=0.10000000149011612] (28.,-3.) -- (28.,-5.) -- (29.,-5.) -- (29.,-3.) -- cycle;
			\fill[line width=2.pt,color=zzttqq,fill=zzttqq,fill opacity=0.10000000149011612] (30.,-2.) -- (30.,-4.) -- (31.,-4.) -- (31.,-2.) -- cycle;
			\fill[line width=2.pt,color=zzttqq,fill=zzttqq,fill opacity=0.10000000149011612] (32.,-1.) -- (32.,-3.) -- (33.,-3.) -- (33.,-1.) -- cycle;
			\fill[line width=2.pt,color=zzttqq,fill=zzttqq,fill opacity=0.10000000149011612] (25.,-7.) -- (25.,-9.) -- (26.,-9.) -- (26.,-7.) -- cycle;
			\fill[line width=2.pt,color=zzttqq,fill=zzttqq,fill opacity=0.10000000149011612] (27.,-6.) -- (27.,-8.) -- (28.,-8.) -- (28.,-6.) -- cycle;
			\fill[line width=2.pt,color=zzttqq,fill=zzttqq,fill opacity=0.10000000149011612] (29.,-5.) -- (29.,-7.) -- (30.,-7.) -- (30.,-5.) -- cycle;
			\fill[line width=2.pt,color=zzttqq,fill=zzttqq,fill opacity=0.10000000149011612] (31.,-4.) -- (31.,-6.) -- (32.,-6.) -- (32.,-4.) -- cycle;
			\fill[line width=2.pt,color=zzttqq,fill=zzttqq,fill opacity=0.10000000149011612] (33.,-3.) -- (33.,-5.) -- (34.,-5.) -- (34.,-3.) -- cycle;
			\draw [line width=2.pt] (22.,-6.)-- (34.,0.);
			\draw [line width=2.pt] (21.,-9.)-- (35.,-2.);
			\draw [line width=2.pt] (22.,-11.)-- (36.,-4.);
			\draw [line width=2.pt,color=zzttqq] (24.,-5.)-- (24.,-7.);
			\draw [line width=2.pt,color=zzttqq] (24.,-7.)-- (25.,-7.);
			\draw [line width=2.pt,color=zzttqq] (25.,-7.)-- (25.,-5.);
			\draw [line width=2.pt,color=zzttqq] (25.,-5.)-- (24.,-5.);
			\draw [line width=2.pt,color=zzttqq] (26.,-4.)-- (26.,-6.);
			\draw [line width=2.pt,color=zzttqq] (26.,-6.)-- (27.,-6.);
			\draw [line width=2.pt,color=zzttqq] (27.,-6.)-- (27.,-4.);
			\draw [line width=2.pt,color=zzttqq] (27.,-4.)-- (26.,-4.);
			\draw [line width=2.pt,color=zzttqq] (23.,-8.)-- (23.,-10.);
			\draw [line width=2.pt,color=zzttqq] (23.,-10.)-- (24.,-10.);
			\draw [line width=2.pt,color=zzttqq] (24.,-10.)-- (24.,-8.);
			\draw [line width=2.pt,color=zzttqq] (24.,-8.)-- (23.,-8.);
			\draw [line width=2.pt,color=zzttqq] (22.,-6.)-- (22.,-8.);
			\draw [line width=2.pt,color=zzttqq] (22.,-8.)-- (23.,-8.);
			\draw [line width=2.pt,color=zzttqq] (23.,-8.)-- (23.,-6.);
			\draw [line width=2.pt,color=zzttqq] (23.,-6.)-- (22.,-6.);
			\draw [line width=2.pt,color=zzttqq] (28.,-3.)-- (28.,-5.);
			\draw [line width=2.pt,color=zzttqq] (28.,-5.)-- (29.,-5.);
			\draw [line width=2.pt,color=zzttqq] (29.,-5.)-- (29.,-3.);
			\draw [line width=2.pt,color=zzttqq] (29.,-3.)-- (28.,-3.);
			\draw [line width=2.pt,color=zzttqq] (30.,-2.)-- (30.,-4.);
			\draw [line width=2.pt,color=zzttqq] (30.,-4.)-- (31.,-4.);
			\draw [line width=2.pt,color=zzttqq] (31.,-4.)-- (31.,-2.);
			\draw [line width=2.pt,color=zzttqq] (31.,-2.)-- (30.,-2.);
			\draw [line width=2.pt,color=zzttqq] (32.,-1.)-- (32.,-3.);
			\draw [line width=2.pt,color=zzttqq] (32.,-3.)-- (33.,-3.);
			\draw [line width=2.pt,color=zzttqq] (33.,-3.)-- (33.,-1.);
			\draw [line width=2.pt,color=zzttqq] (33.,-1.)-- (32.,-1.);
			\draw [line width=2.pt,color=zzttqq] (25.,-7.)-- (25.,-9.);
			\draw [line width=2.pt,color=zzttqq] (25.,-9.)-- (26.,-9.);
			\draw [line width=2.pt,color=zzttqq] (26.,-9.)-- (26.,-7.);
			\draw [line width=2.pt,color=zzttqq] (26.,-7.)-- (25.,-7.);
			\draw [line width=2.pt,color=zzttqq] (27.,-6.)-- (27.,-8.);
			\draw [line width=2.pt,color=zzttqq] (27.,-8.)-- (28.,-8.);
			\draw [line width=2.pt,color=zzttqq] (28.,-8.)-- (28.,-6.);
			\draw [line width=2.pt,color=zzttqq] (28.,-6.)-- (27.,-6.);
			\draw [line width=2.pt,color=zzttqq] (29.,-5.)-- (29.,-7.);
			\draw [line width=2.pt,color=zzttqq] (29.,-7.)-- (30.,-7.);
			\draw [line width=2.pt,color=zzttqq] (30.,-7.)-- (30.,-5.);
			\draw [line width=2.pt,color=zzttqq] (30.,-5.)-- (29.,-5.);
			\draw [line width=2.pt,color=zzttqq] (31.,-4.)-- (31.,-6.);
			\draw [line width=2.pt,color=zzttqq] (31.,-6.)-- (32.,-6.);
			\draw [line width=2.pt,color=zzttqq] (32.,-6.)-- (32.,-4.);
			\draw [line width=2.pt,color=zzttqq] (32.,-4.)-- (31.,-4.);
			\draw [line width=2.pt,color=zzttqq] (33.,-3.)-- (33.,-5.);
			\draw [line width=2.pt,color=zzttqq] (33.,-5.)-- (34.,-5.);
			\draw [line width=2.pt,color=zzttqq] (34.,-5.)-- (34.,-3.);
			\draw [line width=2.pt,color=zzttqq] (34.,-3.)-- (33.,-3.);
			\begin{scriptsize}
			\draw [fill=black] (22.,-6.) circle (6 pt);
			\draw [fill=black] (24.,-5.) circle (6pt);
			\draw [fill=black] (26.,-4.) circle (6pt);
			\draw [fill=black] (23.,-8.) circle (6pt);
			\draw [fill=black] (25.,-7.) circle (6pt);
			\draw [fill=black] (27.,-6.) circle (6pt);
			\draw [fill=black] (28.,-3.) circle (6pt);
			\draw [fill=black] (29.,-5.) circle (6pt);
			\draw [fill=black] (30.,-2.) circle (6pt);
			\draw [fill=black] (31.,-4.) circle (6pt);
			\draw [fill=black] (32.,-1.) circle (6pt);
			\draw [fill=black] (33.,-3.) circle (6pt);
			\draw [fill=black] (24.,-10.) circle (6pt);
			\draw [fill=black] (26.,-9.) circle (6pt);
			\draw [fill=black] (28.,-8.) circle (6pt);
			\draw [fill=black] (30.,-7.) circle (6pt);
			\draw [fill=black] (32.,-6.) circle (6pt);
			\draw [fill=black] (34.,-5.) circle (6pt);
			\end{scriptsize}
			\end{tikzpicture}
			\caption{Disjoint $2 \times 3$ grids}\label{fig2}
	\end{figure}
	\begin{construction}{\label{NearGridED}}\textbf{Existence of efficiently dominatable near-Grid graphs:} \hfill\\
	Note that in figure \ref{fig_P11P11}, the intersections of two lines correspond to vertices. The subtructure $P_{11} \Box P_{11}$ of an infinite grid is highlighted using bold lines. 
	Independently examining the grid $ P_{11} \Box P_{11} $, it can be observed that there exists an $ F(P_{11} \Box P_{11}) $-set resulting in voids which lie on the boundaries as shown in figure \ref{fig_P11P11}. These voids can be dominated by adding new vertices, one to dominate each void so that the resultant graph is efficiently dominatable. In general, given a grid $ P_n \Box P_n $, where $ n \ge 7 $, if there $ k $ voids generated by an $ F(P_n \Box P_n) $-set, then by arranging them suitably to lie on the boundaries, we can add $ k $ new vertices and make them adjacent to one  void each.  Then, the resultant graph becomes efficiently dominatable.  This results in a new class of efficiently dominatable graphs which are nearly grid graphs.
\end{construction}

	\subsubsection{Infinite Triangular grid}\label{itg}
	A triangular grid graph is formed by triangular tessellations. 
	We label the $j^{th}$ vertex in the $i^{th}$ row of a triangular grid as $v_{i,j}$ (refer to figure \ref{fig3}). By following the same procedure explained for an infinte rectangular grid in Section \ref{irg}, we can construct an \textit{EDS} for an infinite triangular grid. For ease of reference, we refer to an infinite triangular grid by $T_\infty$.
	\begin{figure}[!h]
		\centering
		\begin{tikzpicture}[line cap=round,line join=round,>=triangle 45,x=1.0cm,y=1.0cm,rotate=90]
		\clip(-1.,-5.) rectangle (10.,5.);
		\draw [line width=2.pt] (0.,0.)-- (4.330127018922193,2.5);
		\draw [line width=2.pt] (4.330127018922193,2.5)-- (8.660254037844386,0.);
		\draw [line width=2.pt] (8.660254037844386,0.)-- (4.330127018922193,-2.5);
		\draw [line width=2.pt] (4.330127018922193,-2.5)-- (0.,0.);
		\draw [line width=2.pt] (0.8660254037844386,-0.5)-- (5.196152422706632,2.);
		\draw [line width=2.pt] (1.7320508075688772,-1.)-- (6.06217782649107,1.5);
		\draw [line width=2.pt] (2.598076211353316,-1.5)-- (6.928203230275509,1.);
		\draw [line width=2.pt] (3.4641016151377544,-2.)-- (7.794228634059947,0.5);
		\draw [line width=2.pt] (5.196152422706632,-2.)-- (0.8660254037844386,0.5);
		\draw [line width=2.pt] (6.06217782649107,-1.5)-- (1.7320508075688772,1.);
		\draw [line width=2.pt] (6.928203230275509,-1.)-- (2.598076211353316,1.5);
		\draw [line width=2.pt] (7.794228634059947,-0.5)-- (3.4641016151377544,2.);
		\draw [line width=2.pt] (0.8660254037844386,-0.5)-- (0.8660254037844386,0.5);
		\draw [line width=2.pt] (1.7320508075688772,-1.)-- (1.7320508075688772,1.);
		\draw [line width=2.pt] (2.598076211353316,-1.5)-- (2.598076211353316,1.5);
		\draw [line width=2.pt] (3.4641016151377544,2.)-- (3.4641016151377544,-2.);
		\draw [line width=2.pt] (4.330127018922193,-2.5)-- (4.330127018922193,2.5);
		\draw [line width=2.pt] (5.196152422706632,-2.)-- (5.196152422706632,2.);
		\draw [line width=2.pt] (6.06217782649107,-1.5)-- (6.06217782649107,1.5);
		\draw [line width=2.pt] (6.928203230275509,-1.)-- (6.928203230275509,1.);
		\draw [line width=2.pt] (7.794228634059947,-0.5)-- (7.794228634059947,0.5);
		\draw [line width=2.pt] (3.4641016151377544,0.)-- (3.4641016151377544,-1.);
		\draw [line width=2.pt] (2.598076211353316,-0.5)-- (3.4641016151377544,0.);
		\draw [line width=2.pt] (2.598076211353316,-0.5)-- (3.4641016151377544,0.);
		\draw [line width=2.pt] (2.598076211353316,-0.5)-- (3.4641016151377544,0.);
		\draw [line width=2.pt] (2.598076211353316,-0.5)-- (3.4641016151377544,0.);
		\draw [line width=2.pt] (3.4641016151377544,0.)-- (3.4641016151377544,-1.);
		\draw [line width=2.pt] (4.330127018922193,-0.5)-- (2.598076211353316,0.5);
		\draw [line width=2.pt] (3.4641016151377544,1.)-- (2.598076211353316,1.5);
		\draw [line width=2.pt] (3.4641016151377544,0.)-- (3.4641016151377544,-1.);
		\draw [line width=2.pt] (3.4641016151377544,-1.)-- (4.330127018922193,-0.5);
		\begin{scriptsize}
		\draw [fill=black] (2.598076211353316,-0.5) circle (2.5pt);
		\draw [fill=black] (3.4641016151377544,0.) circle (2.5pt);
		\draw [fill=black] (3.4641016151377544,-1.) circle (2.5pt);
		\draw [fill=black] (4.330127018922193,0.5) circle (2.5pt);
		\draw [fill=black] (4.330127018922193,-0.5) circle (2.5pt);
		\draw [fill=black] (2.598076211353316,0.5) circle (2.5pt);
		\draw [fill=black] (2.598076211353316,1.5) circle (2.5pt);
		\draw[color=black] (2.7065679169779537,1.8027641849391944) node[xshift=-6pt] {{\large{$v_{8,1}$}}};
		\draw [fill=black] (3.4641016151377544,1.) circle (6pt);
		\draw [fill=black] (4.330127018922193,-1.5) circle (2.5pt);
		\draw [fill=black] (3.4641016151377544,-2.) circle (2.5pt);
		\draw[color=black] (3.572384361983194,-2.4086025086931707) node[xshift=+3pt] {{\large{$v_{7,5}$}}};
		\draw [fill=black] (3.4641016151377544,2.) circle (2.5pt);
		\draw[color=black] (3.572384361983194,2.2998069589236847) node[xshift=-6pt] {{\large{$v_{7,1}$}}};
		\draw [color=black] (4.330127018922193,-2.5) circle (6pt);
		\draw[color=black] (4.438200806988434,-2.8056452826776605) node[xshift=+3pt] {\large{$v_{6,6}$}};
		\draw [fill=black] (4.330127018922193,1.5) circle (2.5pt);
		\draw [color=black] (4.330127018922193,2.5) circle (6pt);
		\draw[color=black] (4.438200806988434,2.7968497329081745) node[xshift=-6pt] {\large{$v_{6,1}$}};
		\draw [fill=black] (5.196152422706632,-2.) circle (2.5pt);
		\draw[color=black] (5.304017251993674,-2.2086025086931707) node[xshift=+3pt] {\large{$v_{5,5}$}};
		\draw [fill=black] (5.196152422706632,-1.) circle (6pt);
		\draw [fill=black] (5.196152422706632,0.) circle (2.5pt);
		\draw [fill=black] (5.196152422706632,1.) circle (2.5pt);
		\draw [fill=black] (5.196152422706632,2.) circle (2.5pt);
		\draw[color=black] (5.304017251993674,2.2998069589236847) node[xshift=-6pt] {\large{$v_{5,1}$}};
		\draw [fill=black] (6.06217782649107,-1.5) circle (2.5pt);
		\draw[color=black] (6.169833696998914,-1.795526096838213) node[xshift=+3pt] {\large{$v_{4,4}$}};
		\draw [fill=black] (6.06217782649107,-0.5) circle (2.5pt);
		\draw [fill=black] (6.06217782649107,0.5) circle (2.5pt);
		\draw [fill=black] (6.06217782649107,1.5) circle (6pt);
		\draw[color=black] (6.169833696998914,1.8027641849391944) node[xshift=-7pt,yshift=7pt] {\large{$v_{4,1}$}};
		\draw [fill=black] (6.928203230275509,-1.) circle (2.5pt);
		\draw[color=black] (7.075734236680323,-1.2583992281775546) node[xshift=+3pt] {\large{$v_{3,3}$}};
		\draw [fill=black] (6.928203230275509,0.) circle (2.5pt);
		\draw [fill=black] (1.7320508075688772,0.) circle (2.5pt);
		\draw [fill=black] (1.7320508075688772,1.) circle (2.5pt);
		\draw[color=black] (1.8808355666488823,1.3297718677604056) node[xshift=-6pt] {\large{$v_{9,1}$}};
		\draw [fill=black] (0.8660254037844386,0.5) circle (6pt);
		\draw[color=black] (1.0150191216436422,0.927290937759155) node[xshift=-6pt] {\large{$v_{10,1}$}};
		\draw [fill=black] (7.794228634059947,-0.5) circle (6pt);
		\draw[color=black] (7.941550681685563,-0.9613564541930646) node[xshift=+3pt] {\large{$v_{2,2}$}};
		\draw [fill=black] (8.660254037844386,0.) circle (2.5pt);
		\draw[color=black] (8.807367126690803,0.3356863197914255) node[xshift=-3pt] {\large{$v_{1,1}$}};
		\draw [fill=black] (7.794228634059947,0.5) circle (2.5pt);
		\draw[color=black] (7.941550681685563,0.8327290937759155) node [xshift=-3pt]{\large{$v_{2,1}$}};
		\draw [fill=black] (6.928203230275509,1.) circle (2.5pt);
		\draw[color=black] (7.075734236680323,1.3297718677604056) node[xshift=-3pt]{\large{$v_{3,1}$}};
		\draw [fill=black] (2.598076211353316,-1.5) circle (6pt);
		\draw[color=black] (2.746652011654122,-2.086025086931707) node {\large{$v_{8,4}$}};
		\draw [fill=black] (1.7320508075688772,-1.) circle (2.5pt);
		\draw[color=black] (1.8808355666488823,-1.4583992281775546) node[xshift=+3pt] {\large{$v_{9,3}$}};
		\draw [fill=black] (0.8660254037844386,-0.5) circle (2.5pt);
		\draw[color=black] (1.0150191216436422,-1.0613564541930646) node[xshift=+3pt]{\large{$v_{10,2}$}};
		\draw [fill=black] (0.,0.) circle (2.5pt);
		\draw[color=black] (-0.4150191216436422,0.0613564541930646) node {\large{$v_{11,1}$}};;
		\end{scriptsize}
		\end{tikzpicture}
		\caption{Labeling of a Triangular Grid}
		\label{fig3}
	\end{figure}
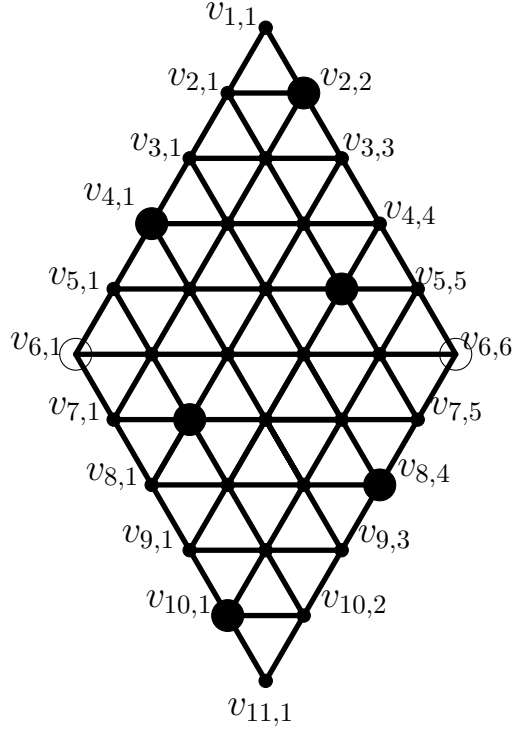	
%	In the case of infinite triangular girds, for a given vertex labeling, there can be multiple interpretations regarding its position. Figure \ref{fig5} will help to avoid this.
	\begin{theorem}
		An infinite triangular Grid is efficiently dominatable.
	\end{theorem}	
	\begin{proof}
		We construct an EDS, say $S$, of $T_\infty $ as follows: Choose an arbitrary vertex, say $v_{i,j}$, and let $v_{i,j} \in S$. Select the four vertices  $v_{i-1,j+2}$, $v_{i+1,j-2}$, $v_{i+2,j}$ and $v_{i-2,j}$ around  $v_{i,j}$ as shown in figure \ref{fig5}. Clearly, these four vertices are at distance three from $v_{i,j}$ and are at mutually at distance at least three (refer to figure \ref{fig5}). Hence, the set $ S = S \cup \{v_{i-1,j+2}, v_{i+1,j-2}, v_{i+2,j}, v_{i-2,j}\}$ will be a $ 2 $-packing of $ T_{\infty} $. Next, for each vertiex $ v_{p, q} \in S$, choose another set of four vertices in the same manner and add them to $S$. The chosen vertices are in such a way that they are mutually at distance at least three and they all fall on the diagonal lines as shown in the figure \ref{fig4}. Hence, the set $ S $ so generated forms a $ 2 $-packing of $ T_{\infty} $.  Further, pairing the vertices of consecutive diagonal lines to form opposite corners of disjoint $3 \times 2$ triangular grids (that is, grids containing 3 rows with each row containing 2 vertices), it can be observed that these vertices dominate every vertex between the diagonal lines and no voids are created. As the grid is infinite, it is possible to iteratively continue this pattern of choosing vertices  along all directions and add to $ S $. Based on the way the vertices are selected, it can be observed the set $ S $ obtained at each iteration is a $ 2 $-packing of $ T_{\infty} $ and all vertices between the diagonal lines are dominated (refer to figure \ref{fig4}).  Hence, the final set $S$ so obtained will be an \textit{EDS} of $ T_{\infty} $. 
	\end{proof}
	\definecolor{ffffff}{rgb}{1.,1.,1.}
	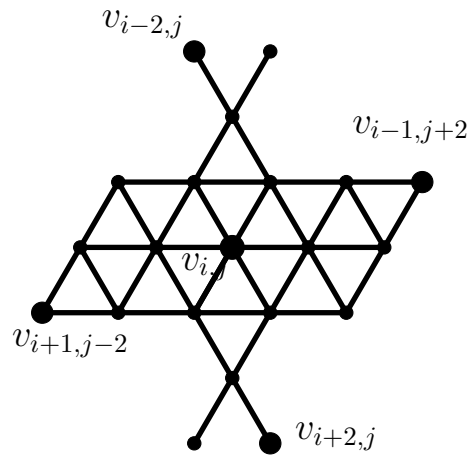
\begin{figure}[!h]
		\centering
		\begin{tikzpicture}[line cap=round,line join=round,>=triangle 45,x=1.0cm,y=1.0cm,rotate=90]
		\clip(6.42237,-3.084296711794365) rectangle (12.706971043452594,3.572538769139867);
		\draw [line width=2.pt] (8.660254037844386,3.)-- (8.660254037844386,-1.);
		\draw [line width=2.pt] (6.928203230275509,0.)-- (10.392304845413264,2.);
		\draw [line width=2.pt] (9.526279441628825,2.5)-- (9.526279441628825,-1.5);
		\draw [line width=2.pt] (10.392304845413264,2.)-- (10.392304845413264,-2.);
		\draw [line width=2.pt] (8.660254037844386,-1.)-- (10.392304845413264,-2.);
		\draw [line width=2.pt] (8.660254037844386,0.)-- (10.392304845413264,-1.);
		\draw [line width=2.pt] (8.660254037844386,1.)-- (10.392304845413264,0.);
		\draw [line width=2.pt] (8.660254037844386,2.)-- (10.392304845413264,1.);
		\draw [line width=2.pt] (8.660254037844386,3.)-- (10.392304845413264,2.);
		\draw [line width=2.pt] (10.392304845413264,0.)-- (12.12435565298214,1.);
		\draw [line width=2.pt] (7.794228634059947,0.5)-- (8.660254037844386,0.);
		\draw [line width=2.pt] (10.392304845413264,1.)-- (11.258330249197702,0.5);
		\draw [line width=2.pt] (7.794228634059947,0.5)-- (6.928203230275509,1.);
		\draw [line width=2.pt] (11.258330249197702,0.5)-- (12.12435565298214,0.);
		\draw [line width=2.pt] (8.660254037844386,1.)-- (9.526279441628825,0.5);
		\draw [line width=2.pt] (8.660254037844386,0.)-- (10.392304845413264,1.);
		\draw [line width=2.pt] (8.660254037844386,-1.)-- (10.392304845413264,0.);
		\draw [line width=2.pt] (8.660254037844386,2.)-- (9.526279441628825,2.5);
		\draw [line width=2.pt] (9.526279441628825,-1.5)-- (10.392304845413264,-1.);
		\begin{scriptsize}
		\draw [fill=black] (9.526279441628825,0.5) circle (4.5pt);
		\draw[color=black] (9.620057225905109,0.85395510257168) node[yshift=-10pt] {\large{$v_{i,j}$}};
		\draw [fill=black] (8.660254037844386,1.) circle (2.5pt);
		\draw [fill=black] (8.660254037844386,0.) circle (2.5pt);
		\draw [fill=black] (7.794228634059947,0.5) circle (2.5pt);
		\draw [fill=black] (10.392304845413264,1.) circle (2.5pt);
		\draw [fill=black] (10.392304845413264,0.) circle (2.5pt);
		\draw [fill=black] (11.258330249197702,0.5) circle (2.5pt);
		\draw [fill=black] (12.12435565298214,1.) circle (4.0pt);
		\draw[color=black] (11.714417686259077,1.3196095867573205) node[xshift=-10pt, yshift=22pt] {\large{$v_{i-2,j}$}};
		\draw [fill=black] (6.928203230275509,0.) circle (4.0pt);
		\draw[color=black] (7.0256967655511415,0.32177854921666205) node[xshift=34pt] {\large{$v_{i+2,j}$}};
		\draw [fill=black] (9.526279441628825,1.5) circle (2.5pt);
		\draw [fill=black] (8.660254037844386,2.) circle (2.5pt);
		\draw [fill=black] (9.526279441628825,2.5) circle (2.5pt);
		\draw [fill=black] (10.392304845413264,2.) circle (2.5pt);
		\draw [fill=black] (8.660254037844386,3.) circle (4.0pt);
		\draw[color=black] (8.25527040578712,2.6285760756725127) node {\large{$v_{i+1,j-2}$}};
		\draw [fill=black] (9.526279441628825,-0.5) circle (2.5pt);
		\draw [fill=black] (9.526279441628825,-1.5) circle (2.5pt);
		\draw [fill=black] (10.392304845413264,-2.) circle (4.0pt);
		\draw[color=black] (11.084844046023099,-1.6738835258646547) node[xshift=5pt]{\large{$v_{i-1,j+2}$}};
		\draw [fill=black] (10.392304845413264,-1.) circle (2.5pt);
		\draw [fill=black] (8.660254037844386,-1.) circle (2.5pt);
		\draw [fill=black] (6.928203230275509,1.) circle (2.5pt);
		\draw [fill=black] (12.12435565298214,0.) circle (2.5pt);
		\end{scriptsize}
		\end{tikzpicture}		
		\caption{Choice of vertices in an infinite triangular grid}
		\label{fig5}
		% \vspace{1cm}
	\end{figure}	
	\begin{figure}[!h]
		\centering
		\includegraphics[angle=90,scale=0.85,height=8.5cm]{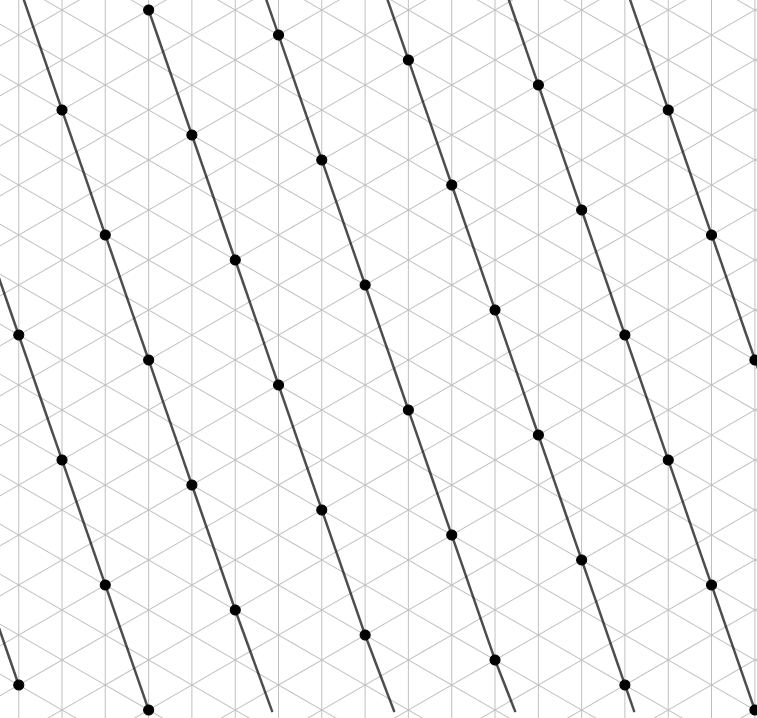}
		\caption{EDS of an Infinite triangular grid}
		\label{fig4}
	\end{figure}
	
	\subsubsection{Infinite hexagonal grid} \label{hex1}
	A hexagonal grid graph is formed by tessellations of hexagons.
	We know that $C_6$ is an efficiently dominatable graph and any two diagonally opposite vertices form an \textit{EDS} of $ C_6 $. This property forms the basis for constructing an EDS for an infinite hexagonal grid.  For conveninece, we use $ H_{\infty} $ to refer to an infinite hexagonal grid.
	\begin{figure}[!h]
		\centering
		\includegraphics[scale=0.7]{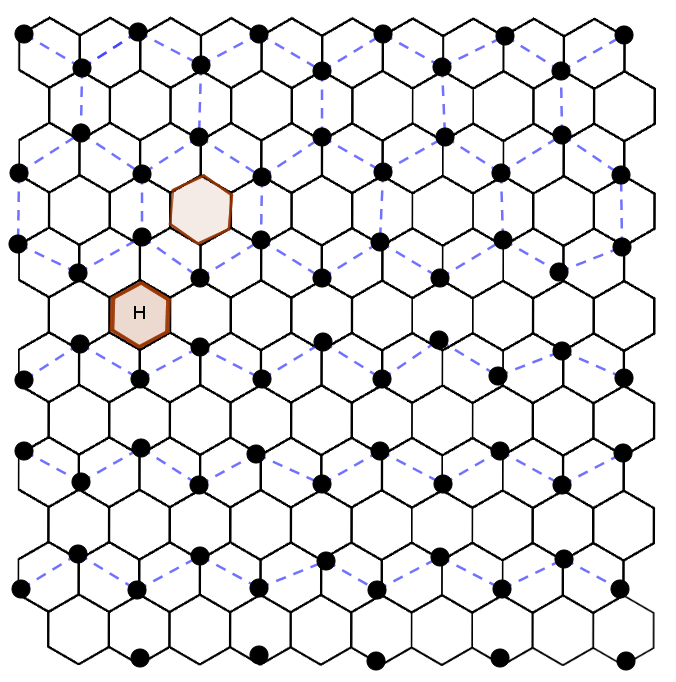}
		\caption{EDS of an infinite hexagonal  grid}
		\label{fig8}
	\end{figure}
	\begin{theorem}
		An infinite hexagonal grid is efficiently dominatable.		
	\end{theorem}	
	\begin{proof}
		We construct an EDS, say $S$, of $ H_{\infty} $ as follows: Note that each vertex in an infinite hexagonal grid lies in (or common to) three adjacent hexagons. Choose an arbitrary vertex, say $v$, and let $v \in S$. Next, from each of the three hexagons to which $ v $ belongs, select the vertices which are at distance $3$ from $v$ (that is, the vertex diagonally opposite to $ v $ in each hexagon containing $ v $). Add these three vertices to $S$. Note that the set $ S $ generated at this stage is a $ 2 $-packing as all its vertices are mutually at distance at least three in $ H_{\infty} $ (refer to figure \ref{fig8}). Next, for each of the newly added vertices, repeat the process of choosing the diagonally opposite vertices from the hexagons they belong to. This process of constructing a $2$-packing for $ H_{\infty} $ results in a structure as shown in figure \ref{fig8}. It can be observed that all those vertices in $ S $ are mutually at distance at least three and they lie on the zig-zag lines. It can be noted that for any hexagon, either two diagonally opposite vertices lie on these lines or no vertex lies on these lines. Suppose no vertex of a hexagon $H$ belongs to $S$, then each vertex of $H$ is dominated by a unique neighbor outside $H$.  Thereby, the set $ S $ so generated dominates all vertices of $ H_{\infty} $ and hence, $S$ is a \textit{EDS} of $ H_{\infty} $.		
	\end{proof}
	\section{Conclusion}
	In this paper, the concept of efficient domination has been studied on lattice graphs, namely rectangular grid graphs, triangular grid graphs, and hexagonal grid graphs. 
	A characterization is obtained for efficiently dominatable finite rectangular grids. A finite square grid $ P_n \Box P_n $ has been shown to efficiently dominatable if and only if $ n = 4 $. For those finite square grids which are not efficiently dominatable, a lower bound on its efficient domination number is derived.  A contructive procedure is given to derived these lower bounds and the construction could be extended to both infinite rectangular grids and infinite triangular grids to prove that they are efficiently dominatable.  The lower bound derived is found to be attained at most values of $ n $, based on which a conjecture is stated. Another constructive procedure has been discussed to study the notion of efficient domination in infinite hexagonal grids. The study on finite triangular and hexagonal grid structures is in progress.  
	%\section*{References}
	%\bibliographystyle{elsarticle-num}
	%\bibliographystyle{apa}
	%\bibliography{literaturesurvey_product}

	%\begin{acknowledgements}
	%\end{acknowledgements}


	

	%\begin{thebibliography}{}
	%%
	%% and use \bibitem to create references. Consult the Instructions
	%% for authors for reference list style.
	%\end{thebibliography}
	
\end{document}